\documentclass[reqno,12pt]{amsart}%
\usepackage[body={17cm,21cm}]{geometry}

\usepackage[T1]{fontenc}
\usepackage{graphicx}
\usepackage{mathrsfs}
\usepackage{amscd,latexsym,amsthm,amsfonts,amssymb,amsmath,amsxtra}
\usepackage[colorlinks, urlcolor=blue,  citecolor=blue]{hyperref}
\usepackage[all]{xy}%
\providecommand{\U}[1]{\protect\rule{.1in}{.1in}}
\RequirePackage{amsmath}
\RequirePackage{amssymb}

\newcommand{\BA}{{\mathbb {A}}}

\newcommand{\BG}{{\mathbb {G}}}

\newcommand{\BQ}{{\mathbb {Q}}}

\newcommand{\BZ}{{\mathbb {Z}}}

\newcommand{\RA}{{\mathbf {A}}}

\newcommand{\Br}{{\mathrm{Br}}}

\newcommand{\coker}{{\mathrm{coker}}}

\newcommand{\Div}{{\mathrm{Div}}}

\newcommand{\Gal}{{\mathrm{Gal}}}

\newcommand{\Hom}{{\mathrm{Hom}}}

\newcommand{\Ind}{{\mathrm{Ind}}}

\newcommand{\inv}{{\mathrm{inv}}}

\newcommand{\ord}{{\mathrm{ord}}}

\newcommand{\Pic}{\mathrm{Pic}}

\renewcommand{\mod}{\ \mathrm{mod}\ }

\newcommand{\Res}{{\mathrm{Res}}}

\newcommand{\Spec}{{\mathrm{Spec}}}

\newcommand{\St}{{\mathrm{St}}}

\font\cyr=wncyr10
\newcommand{\Sha}{\hbox{\cyr X}}

\newcommand{\sbt}{\subset}

\numberwithin{equation}{section}

\theoremstyle{remark}

\newtheorem{defi}{\rm{\textbf{Definition}}}[section]

\newtheorem{exam}[defi]{\rm{\textbf{Example}}}

\newtheorem{rem}[defi]{\rm{\textbf{Remark}}}

\theoremstyle{plain}

\newtheorem{thm}[defi]{\rm{\textbf{Theorem}}}

\newtheorem{cor}[defi]{\rm{\textbf{\textbf{Corollary}}}}

\newtheorem{lem}[defi]{\rm{\textbf{Lemma}}}

\newtheorem{prop}[defi]{\rm{\textbf{\textbf{Proposition}}}}

\newtheorem*{hypo}{\rm{\textbf{Hypothesis}}}

\begin{document}

\title[]
{Strong approximation for a family of norm varieties}

\author{Yang Cao}

\address{Yang Cao: Department of Mathematics and Physics, Leibniz University Hannover, Germany}

\email{yangcao1988@gmail.com}

\author{Dasheng Wei}

\address{Dasheng Wei: Academy of Mathematics and System Science \\ Chinese Academy of Sciences \\ Beijing 100190,  China}

\email{dshwei@amss.ac.cn}

\author{Fei Xu}

\address{Fei Xu: School of Mathematical Sciences, Capital Normal University, Beijing 100048, China}

\email{xufei@math.ac.cn}



\begin{abstract}   We study strong approximation of the equation $$N_{L/k}(x) = \prod_{i=1}^n p_i(t)$$ where $L/k$ is a finite extension of number fields and $p_i(t)$'s are non-proportional irreducible polynomials over $k$. We prove this equation  satisfies strong approximation with Brauer-Manin obstruction when 
$L$ can be imbedded in $k[t]/(p_i(t))$ over $k$ for all $1\leq i\leq n$. Under Schinzel's hypothesis, we prove that the same result is true without assuming that $L$ can be imbedded in $k[t]/(p_i(t))$ for all $1\leq i\leq n$ when $L/k$  is cyclic.
 \end{abstract}

\maketitle


\section{Introduction}

 Weak approximation of the following family of norm varieties
\begin{equation} \label{n-eq} N_{L/k}(x) =q(t)   \end{equation}  has been studied extensively, where $L/k$ is a finite extension of number fields and $q(t)$ is a polynomial over $k$. 
So far, there are two approaches to study this problem. The first approach is to study weak approximation of certain torsors over equation (\ref{n-eq}), which depends heavily on the explicit description of the torsors and the descent theory developed in \cite{CTS87}, by

$\bullet$ geometric methods ($k$-rationality, $[L:k]\leq 3$ and $\deg(q(t))\leq 6$; see \cite{CTSSSD}, \cite{SD99}, \cite{CTSa});  

$\bullet$  the circle method (all irreducible factors of $q(t)$ are at most two distinct linear polynomials; see \cite{HBS02}, \cite{CTHS03}, \cite{Jo13}, \cite{SS14});

$\bullet$ the sieve method ($k=\Bbb Q$ and $q(t)$ is irreducible with $\deg (q(t))\leq 3$; see \cite{BHB}, \cite{DSW}, \cite{Irv});

$\bullet$ additive combinatorics ($k=\Bbb Q$ and $q(t)$ is a product of linear polynomials; see  \cite{BM13}, \cite{BMS}). 

Another approach is the fibration method which can be traced back to Hasse's proof about the local-global principle of quadratic forms. Along the same spirit of Hasse's proof, 
Colliot-Th\'el\`ene and Sansuc in \cite{CTSS} proved the equation (\ref{n-eq}) satisfies weak approximation by assuming Schinzel's hypothesis when $k=\Bbb Q$, $[L:k]=2$ and $q(t)$ is an irreducible polynomial over $\Bbb Q$. Such a conditional result (under Schinzel's hypothesis) was largely extended in \cite{CTSKSD} and \cite{Wei}. In \cite{HSW}, Harpaz, Skorobogatov and Wittenberg explained how to replace Schinzel's hypothesis with the recent achievement in additive combinatorics developed by Green, Tao and Ziegler to obtain the unconditional results.  In \cite{HW}, Harpaz and Wittenberg further improved the fibration method such that most of the above mentioned unconditional results can be covered. Moreover, they also proposed a conjecture in \cite{HW} which implies that (\ref{n-eq}) satisfies weak approximation with Brauer-Manin obstruction in general. 
 
For strong approximation of equation (\ref{n-eq}), the first non-trivial example was given by Derenthal and Wei in \cite{DW} where $[L;k]=4$ and $q(t)$ is an irreducible quadratic polynomial which has a root in $K$.  By using the explicit description of the universal torsor of (\ref{n-eq}) in this special case,  they proved that this universal torsor satisfies strong approximation in \cite{DW}. Under Schinzel's hypothesis, the related results for integral points and strong approximation of Ch\^{a}telete surfaces are studied in \cite{Gun} and \cite{MIT}. 

In \cite{CTH}, Colliot-Th\'el\`ene and Harari developed the fibration method for studying strong approximation. However, their method can not applied to study the equation (\ref{n-eq}). As pointed out in \cite[P.175]{CTH}, there are two difficulties for proving strong approximation of the equation (\ref{n-eq}) by fibration method. 

D1: non-trivial elements in Brauer groups of special fibres over rational points are infinite. 

D2:  the fibers over $q(t)$ may not split.

In this paper, we will overcome the first difficulty (D1) by using the natural action of torus $\Res_{L/k}^1(\Bbb G_m)$. For any given open set of adelic points of (\ref{n-eq}), one can find a finite subgroup of Brauer group of the generic fiber via this action such that the existence of rational points over almost all fibers of rational points can be tested by using the restriction of this finite group to these special fibers (see Proposition \ref{assum}). 
It is well known that one can overcome the second difficulty (D2) by using Schinzel's hypothesis. However, one needs to get rid of obstruction from finite extensions from residue fields of fibers over $p(t)$ to their algebraic closure in their function fields. 
This new ingredient forces us to restrict the field extension $L/k$, for example to be cyclic. Over $\Bbb P^1$, Harpaz and Wittenberg in \cite{HW} proposed a conjecture to replace Schinzel hypothesis to solve the similar difficulty as (D2). One can also raise an integral version of Harpaz-Wittenberg conjecture to establish strong approximation of equation (\ref{n-eq}) for any finite field extension $L/k$.

Notation and terminology are standard. Let $k$ be a number field, $\Omega_k$ the set of all primes in $k$ and  $\infty_k$ the set of all Archimedean primes in $k$. Write $v<\infty_k$ for $v\in \Omega_k\setminus \infty_k$. Let $\frak o_k$ be the ring of integers of $k$ and $\frak o_{k,S}$ the $S$-integers of $k$ for a finite set $S$ of $\Omega_k$ containing $\infty_k$. For each $v\in \Omega_k$, the completion of $k$ at $v$ is denoted by $k_v$, the completion of $\frak o_k$ at $v$ by $\frak o_{k_v}$ and the residue field at $v$ by $k(v)$ for $v<\infty_k$. Write $\frak o_{k_v}=k_v$ for $v\in \infty_k$ and $k_{\infty}=\prod_{v\in \infty_k} k_v$. Let ${\mathbf A}_k$ be the adele ring of $k$ and ${\mathbf A}_k^f$ the finite adele ring of $k$. For a finite extension $L/k$ of number fields and a finite set $S$ of $\Omega_k$ containing $\infty_k$, the set of elements in $L$ which are integral over all primes not lying above $S$ is denoted by $\frak o_{L, S}$.

A variety $X$ over $k$ is defined to be a reduced separated scheme of finite type over $k$. We denote $X_{\bar k}=X\times_k \bar{k}$ with $\bar{k}$ a fixed algebraic closure of $k$.  Let $$\Br(X)=H_{\text{\'et}}^2(X, \Bbb G_m) , \ \  \ \ \ \Br_1(X)= \ker[\Br(X)\rightarrow \Br(X_{\bar{k}})] $$ \
and
 $$ \Br_a(X)={\rm {coker}}[\Br (k) \rightarrow \Br_1(X)]  , \ \ \ \ \ \  \overline{\Br(X)} ={\rm {coker}}[\Br (k) \rightarrow \Br(X)]  .$$


For any subset $B$ of $\Br(X)$, we set 
$$ X({\mathbf A}_k)^B = \{ (x_v)_{v\in \Omega_k}\in X({\mathbf A}_k): \ \ \sum_{v\in \Omega_k} \inv_v(\xi(x_v))=0, \ \ \forall \xi\in B \}  . $$ 
We denote by $\pi_0(X(k_{\infty}))$ and $\pi_0(X(k_v))$ the set of connected components of $X(k_{\infty})$ and  $X(k_v)$ for $v\in \infty_k$ with discrete topology respectively.
 Define
$ X(\mathbf A_k)_{\bullet}=  \pi_0(X(k_\infty))\times X(\mathbf A_k^f) $
and 
$$X(\RA_k)_E:=E\times X(\RA_k^f)\sbt X(\RA_k)_{\bullet}\ \ \ \text{and}\ \ \  X(k)_E:=X(k)\cap X(\RA_k)_E\sbt X(\RA_k)_{\bullet}$$
for any subset $E \sbt \pi_0(X(k_{\infty}))$.  Since an element in $\Br(X)$ takes the constant value at each connected component of $X(k_\infty)$, the set $X(\RA_k)_E^B$ is well-defined. Write $pr_{\bullet}: X(\RA_k)\to X(\RA_k)_{\bullet}$ the natural projection map. If $G$ is a connected algebraic group over $k$, then $G(k_\infty)^+$ is denoted as the connected component of identity of Lie group $G(k_\infty)$.  

Let $\phi: X\rightarrow Y$ be a morphism of schemes. For $y\in Y$, we write $X_y$ for the fiber of $\phi$ over $y$.

\begin{defi} \label{sa}  Let $X$ be a variety over a number field $k$, $B$ a subset of $\Br(X)$ and $S_0$ a finite subset of $\Omega_k$. 
We say $X$ satisfies strong approximation with respect to $B$ off $S_0$ if $X(k)$ is dense in $ pr^{S_0}(X({\mathbf A}_k)^B)\neq \emptyset$ where $pr^{S_0}: X(\RA_k) \rightarrow X(\RA_k^{S_0})$ the natural projection map and $\RA_k^{S_0}$ is the adele ring of $k$ without $S_0$-components. 
\end{defi}

The first main result of this paper is the following theorem (see Corollary \ref{app-neq}).

\begin{thm}\label{intr-m1}
Let $X$ be the smooth locus of the following affine variety  
$$  \prod_{i=1}^m N_{L_i/k}(x_i) = c\prod_{j=1}^n p_j(t)^{e_j} \ \ \ \text{with} \ \ \ c\in k^\times $$ 
where $L_i/k$'s are finite extensions of number fields and $p_j(t)$'s are distinct irreducible monic polynomials over $k$ and $e_j$'s are positive integers. 
If for each $1\leq j \leq n$ there is $L_i$ such that  either $L_i$ can be imbedded in $k[t]/(p_j(t))$ over $k$ or $[L_i: k] \mid e_j$ with $1\leq i\leq m$, then $X$ satisfies strong approximation off $\infty_k$ with respect to $\Br_1(X)$.  
\end{thm}

Under Schinzel's hypothesis, one can remove the assumption that  for each $1\leq j\leq n$ there is $L_i$ such that  either $L_i$ can be imbedded in $k[t]/(p_j(t))$ over $k$ or $[L_i: k] \mid e_j$ with $1\leq i\leq m$ in the above theorem if one of extensions $L_i/k$ for $1\leq i\leq m$ is cyclic (see Corollary \ref{neq-sch}). 

\begin{thm} \label{intr-m2} Let $X$ be the smooth locus of the following affine variety 
$$ \prod_{i=1}^m N_{L_i/k}(x_i) = c\prod_{j=1}^n p_j(t)^{e_j} \ \ \ \text{with} \ \ \ c\in k^\times $$ 
where $L_i/k$'s are finite extensions of number fields and $p_j(t)$'s are distinct irreducible monic polynomials over $k$ and $e_j$'s are positive integers. Assume that one of extensions $L_i/k$ for $1\leq i\leq m$ is cyclic.
Suppose the projection of each connected component of $X(k_\infty)$ to $t$-coordinate is unbounded for all real primes $v\in \infty_k$. If Schinzel's hypothesis holds, then $X$ satisfies strong approximation off $\infty_k$ with respect to $\Br_1(X)$.  
\end{thm}

It should be pointed out that the case of Theorem \ref{intr-m2}  is even unknown for weak approximation before. The second named author further applies this idea to study weak approximation in \cite{Wei1}.

The paper is organized as follows. In \S 2, we work out the local necessary condition at $\infty_k$ of strong approximation for quasi-affine varieties. Then we prove Theorem \ref{intr-m1} and Theorem \ref{intr-m2} in \S 3 and \S4 respectively.  In \S 5, we further apply our method to show certain results in opposite direction of \cite[Corollary 9.10]{HW}.

\section{Necessary conditions at $\infty_k$}

For a number field $k$, it is well known that $k$ is discrete and closed in $\mathbf A_k$ by the product formula. When $X$ is a quasi-affine variety over $k$, then $X(k)$ is discrete and closed in $X(\Bbb A_k)$. Therefore $X(k_\infty)$ is not compact if $X$ satisfies strong approximation off $\infty_k$ in the classical sense. In this section, we extend this necessary condition to strong approximation with Brauer-Manin obstruction when $\overline{\Br(X)}$ is finite. 

\begin{defi} Let $X$ be a variety over a number field $k$ and $B$ be a subgroup of $\Br(X)$.  A connected component $D$ of $X(k_\infty)$ is called admissible with respect to $B$ if $D \subset pr_{\infty_k}(X(\mathbf A_k)^{B})$ where $pr_{\infty_k}: X(\mathbf A_k)\rightarrow X(k_\infty)$ is the projection map. 
\end{defi}

Since each element $b\in \Br(X)$ takes a constant value over a connected component $D_v$ of $X(k_v)$ for $v\in \infty_k$, one can simply write $b(D_v)$ for this constant value. 

Let $B$ be a subgroup of $\Br(X)$. We define $D\sim_B D'$ for two connected components $$D=\prod_{v\in \infty_k}D_v  \ \ \ \text{ and } \ \ \ D'=\prod_{v\in \infty_k} D_v'$$ of $X(k_\infty)$ if
$$ \sum_{v\in \infty_k} b(D_v) =\sum_{v\in \infty_k} b(D_v') $$ for all $b\in B$. This $\sim_B$ provides an equivalent relation among the admissible connected components of $X(k_\infty)$ with respect to $B$.

\begin{lem} Let $B$ be a finite subgroup of $\Br(X)$. Then two admissible connected components $D$ and $D'$ of $X(k_\infty)$ with respect to $B$ satisfy $D\sim_B D'$ if and only if  there is an open compact subset $W\neq \emptyset$ of $X(\mathbf A_k^f)$ such that
$$D\times  W \subset X(\mathbf A_k)^{B}  \ \ \ \text{and} \ \ \ D' \times  W \subset X(\mathbf A_k)^{B} . $$
\end{lem}
\begin{proof} Suppose that $D$ and $D'$ are the admissible connected components of $X(k_\infty)$ with respect to $B$ satisfying $D\sim_B D'$. There are  open compact subsets $W\neq \emptyset$  and $V\neq \emptyset$ in $X(\mathbf A_k^f)$ such that 
$$ D\times W \subset X(\mathbf A_k)^{B}  \ \ \ \ \text{and} \ \ \ \ D'\times V \subset X(\mathbf A_k)^{B} $$ respectively. 
Since $B$ is finite, one can further assume that $$W=\prod_{v< \infty_k} W_v \ \ \ \text{ and } \ \ \ V=\prod_{v< \infty_k} V_v$$  such that each element in $B$ takes a single value over $W_v$ and $V_v$ respectively for each $v<\infty_k$. One can simply denote this single value by $b(W_v)$ and $b(V_v)$ respectively for each $v<\infty_k$. Moreover, one has $b(W_v)=b(V_v)=0$ for almost all $v<\infty_k$. 

Since 
$$ \sum_{v\in \infty_k} b(D_v) =\sum_{v\in \infty_k} b(D_v')  \ \ \ \text{with} \ \ \ D=\prod_{v\in \infty_k} D_v \ \ \text{and} \ \ D'=\prod_{v\in \infty_k} D_v' $$ for all $b\in B$, one concludes that 
$$ \sum_{v< \infty_k} b(W_v) =\sum_{v<\infty_k} b(V_v) $$ for all $b\in B$. This implies that $D'\times W\subset X(\mathbf A_k)^{B} $ as desired. 

Conversely, one takes an element $(x_v)_{v<\infty_k} \in W$ and obtains $D\sim_B D'$.  
\end{proof}

The following proposition gives the necessary condition of strong approximation with Brauer-Manin obstruction for quasi-affine varieties at $\infty_k$ when $\overline{\Br(X)}$ is finite. 

\begin{prop}\label{nec} Let $X$ be a quasi-affine geometrically integral variety over a number field $k$ with $\dim(X)\geq 1$. Suppose that $B$ is a finite subgroup of $\Br(X)$. If $X$ satisfies strong approximation with respect to $B$ off $\infty_k$, then each equivalent class of admissible connected components of $X(k_\infty)$ with respect to $B$ contains a non-compact one.
\end{prop}

\begin{proof} If $k$ has a complex prime $v$, then $X(k_v)$ is connected (see \cite[Chapter 3, Theorem 3.5]{PR}). Since $X$ is quasi-affine, the set $X(k_v)$ is not compact. This implies that every connected component of $X(k_\infty)$ is not compact. Therefore one only needs to consider the case that $k$ is totally real. 

Suppose there is an equivalent class $[D]$ of admissible connected components of $X(k_\infty)$ such that every element in this class is compact. Since $D$ is admissible with respect to $B$, there is an open compact subset $W\neq \emptyset$ of $X(\mathbf A_k^f)$ such that 
$ D\times W \subset X(\mathbf A_k)^{B} $. Since $B$ is finite, one can further assume that $W=\prod_{v<\infty_k} W_v$ such that each element in $B$ takes a constant value over $W_v$ for all $v<\infty_k$. Therefore
$$ (X(k_\infty) \times \prod_{v<\infty_k} W_v ) \cap X(\mathbf A_k)^{B} = (\bigcup_{C\in [D]} C) \times \prod_{v<\infty_k} W_v  . $$
Since the number of connected components of $X(k_\infty)$ is finite by \cite[Chapter 3, Theorem 3.6]{PR}, the set $\bigcup_{C\in [D]} C$ is compact by our assumption. Since $X$ is quasi-affine, one obtains that $X(k)$ is discrete and closed in $X(\mathbf A_k)$. This implies that the set
$$ X(k) \cap  [(\bigcup_{C\in [D]} C) \times \prod_{v<\infty_k} W_v ]= \{ x_1, \cdots, x_l \} $$ is finite. 

Since $X$ is geometrically integral over $k$, there is $v_0< \infty_k$ such that $$V_0=W_{v_0} \setminus \{ x_1, \cdots, x_l\}\neq \emptyset$$ is open and compact in $X(k_{v_0})$ by the Lang-Weil's estimation (see \cite[Theorem 7.7.1]{Po}). Consider an open subset 
$$ X(k_\infty) \times V_0 \times \prod_{v<\infty_k, \ v\neq v_0} W_v $$ of $X(\mathbf A_k)$.  Then  $$[ X(k_\infty) \times V_0 \times \prod_{v<\infty_k, \ v\neq v_0} W_v ] \cap X(\mathbf A_k)^{B}=  (\bigcup_{C\in [D]} C) \times V_0 \times  \prod_{v<\infty_k, \ v\neq v_0} W_v \neq \emptyset $$ contains no points in $X(k)$ any more. This contradicts that $X$ satisfies strong approximation with respect to $B$ off $\infty_k$.  
\end{proof} 

Example 6.2 in \cite{DW} and Example 8.3 in \cite{JaS} can be explained by Proposition \ref{nec}. 

\begin{exam} Let $X$ be a variety over $\Bbb Q$ defined by the equation 
$$x^2+y^2=t(t-2)(t-10) \subset \Bbb A_{\Bbb Q}^3 $$ and the quaternions $\beta_1=(t,-1)$ and $\beta_2=(t-2,-1)$ are the representative of $\overline{\Br(X)}$. Write $B$ the subgroup of $\Br (X)$ generated by $\beta_1$ and $\beta_2$. The set $X(\Bbb R)$ consists of two connected components. One connected component $D_1$ is given by $0\leq t\leq 2$, which is compact. The other one $D_2$ is given by $t\geq 10$, which is not compact. Since $\beta_2$ takes $\frac{1}{2}$ over $D_1$ but takes 0 over $D_2$, one gets $D_1\not\sim_B D_2$. The local points 
$$ (t_v, x_v, y_v) =\begin{cases} (5, x_5, y_5) \ \ \  & v=5 \\
(1,3,0) \ \ \ & v\neq 5 \end{cases} $$ exhibited in \cite[Example 6.2]{DW} imply that $D_1$ is admissible. By Proposition \ref{nec}, one concludes that $X$ does not satisfy strong approximation with Brauer-Manin obstruction off $\infty_k$. 
\end{exam}

\begin{exam} Let $X$ be a Del Pezzo surfaces of degree four defined by the equations 
$$ \begin{cases} x_0 (x_0+x_1) = x_2^2 + (x_2+x_4)^2 \\
(x_0+x_2) (x_0+2x_2) = 2x_1^2 + 3x_3^2 \end{cases} $$ in $\Bbb P^4$ and $U$ be an open subset of $X$ defined by $x_4\neq 0$. As pointed out in the first step of \cite[Example 8.1]{JaS}, the set $U(\Bbb R)$ consists of three connected components $D_0, D_1$ and $D_2$ and one of them is compact with a rational point. Assume this component is $D_0$. Then $D_0$ is admissible. 

Let $B$ be the subgroup of $\Br (U)$ generated by $\alpha_1$ and $\alpha_2$ in the second step of \cite[Example 8.1]{JaS}. We claim that $D_0\not\sim_B D_1$ and $D_0\not\sim_B D_2$. Indeed, suppose  $D_0\sim_B D_1$. Since $\alpha_1$ and $\alpha_2$ are not constant over $U(\Bbb R)$ by the fourth step in \cite[Example 8.1]{JaS}, one obtains that $$\alpha_1(D_0)=\alpha_1(D_1)\neq \alpha_1(D_2) \ \ \ \text{and} \ \ \ \alpha_2(D_0)= \alpha_2(D_1) \neq \alpha_2(D_2) . $$ This implies that 
$$(\alpha_1+\alpha_2)(D_0)=(\alpha_1+\alpha_2)(D_1)=(\alpha_1+\alpha_2)(D_2) $$ which contradicts that $\alpha_1+\alpha_2$ are not constant over $U(\Bbb R)$ in the fifth step in \cite[Example 8.1]{JaS}. Therefore $U$ does not satisfy strong approximation with Brauer-Manin obstruction off $\infty$ by the above claim and Proposition \ref{nec}. 
\end{exam}

To end up this section, we provide a refined version of strong approximation for $\Bbb A^1_k$ which is needed in the next section. There are already several refinements for strong approximation for $\Bbb A_k^1$ implicitly in \cite[Proposition 4.6]{CX1} and \cite[Theorem 6.2]{C}. 

\begin{defi}  Let $v\in \infty_k$ be a real place and $C_v$ be a subset of $k_v$. We say 

$\bullet$ $C_v$ is unbounded above if $\sup \{ x\in C_v \} = + \infty$. 

$\bullet$ $C_v$ is unbounded below if $\inf \{ x\in C_v \} = - \infty$. 

$\bullet$ $C_v$ is unbounded if $C_v$ is either unbounded above or unbounded below.

\end{defi}

\begin{prop} \label{ref-sa} Let $C=\prod_{v\in \infty_k} C_v$ be an open connected subset of $k_\infty$. Suppose that there is either a complex $v\in \infty_k$ with $C_v=\Bbb C$ or a real place $v\in \infty_k$ such that $C_v$ is unbounded. If $W\neq \emptyset$ is an open subset of $\mathbf A^f$, then $k\cap (C\times W)\neq \emptyset$. 
\end{prop}
\begin{proof} One only needs to consider the case that $k$ is totally real and there is $v_0\in \infty_k$ such that $C_{v_0}$ is unbounded. Therefore $C_{v_0}=(a, +\infty)$ or $(-\infty, a)$ for $a\in \Bbb R$. Without loss of generality, one can assume that $W=\prod_{v<\infty_k} W_v$. 

When $k=\Bbb Q$, the set $\Bbb Q\cap (C_{v_0}\times W) \neq \emptyset$ by Dirichlet's prime number theorem with modification on sign if necessary. 

Otherwise,  there is $\epsilon \in O_{k}^\times$ such that 
$$ |\epsilon|_{v_0} >1 \ \ \ \text{ and } \ \ \  |\epsilon|_{v}<1 \ \text{ for all $v\in \infty_k\setminus \{v_0\}$} $$  by \cite[33:8]{OM}. Let $\Sigma$ be a finite subset of $\Omega_k$ containing $\infty_k$ such that $W_v=\frak o_{k_v}$ for all $v \not \in \Sigma$. For each $v< \infty_k$, one can fix $\beta_v\in \frak o_k$ such that $ord_v(\beta_v)>0$ and $ord_{w}(\beta_v)=0$ for all finite $w\neq v$ by finiteness of class number of $\frak o_k$.  By strong approximation for ${\Bbb A}^1$, there is $a\in k$ such that  $$a\in k_{v_0}\times (\prod_{v\in \infty_k\setminus \{v_0\} } C_v) \times W . $$ Let $l_v$ be a sufficiently large integer such that 
$ a +  \beta_v^{l_{v}}\frak o_{k_v} \subseteq W_v$ for each $v\in \Sigma\setminus \infty_k$. Take a sufficiently large positive integer $N$ such that
$$b=a+ \epsilon^{2N+1} \prod_{v\in \Sigma \setminus \infty_k} \beta_v^{l_v} \in C_v $$ for all $v\in \infty_k\setminus \{v_0\}$ and  $b\in (a, +\infty)$ or $(-\infty, a)$ respectively at $v_0$ by replacing $\epsilon$ with $-\epsilon$ if necessary. Therefore $$b\in k\cap (C\times W)$$ as desired.
\end{proof}

\section{Fibration over $\Bbb A_k^1$ with an action of torus}

Let $X$ be a smooth and geometrically integral variety over a number field $k$. Suppose that $X\xrightarrow{f} \Bbb A_k^1$ is a surjective morphism over $k$ with geometrically integral generic fiber. There is an open dense subset $U$ of $\Bbb A_k^1$ over $k$ such that $f|_V: V=f^{-1}(U) \rightarrow U$ is smooth with geometrically integral fibers. Write $$\Bbb A_k^1\setminus U=\{P_1, \cdots, P_n \}$$ where $P_1, \cdots, P_n$ are the closed points over $k$ and $k_i=k(P_i)$ are the residue fields of $P_i$ for $1\leq i\leq n$. 
Let $D_i=f^{-1}(P_i)$ and $\{D_{i,j}\}_{j=1}^{g_i}$ be the set of irreducible components of $D_i$ over $k_i$ for $1\leq i\leq n$. Then one has the exact sequence 
with the residue maps $$  0\rightarrow \Br(X) \rightarrow  \Br(V) \xrightarrow {(\partial_{D_{i,j}})} \bigoplus_{i=1}^n \bigoplus_{j=1}^{g_i} H^1(D_{i,j}, \Bbb Q/\Bbb Z)$$ given by  \cite[(3.9)]{CT92}.
By the functoriality of residue maps and the Faddeev exact sequence (see \cite[\S 1.1 and \S 1.2]{CTSD}), one has the following commutative diagram 
\begin{equation} \label{residue} \xymatrix{\Br_a(U)\ar[r]_-{\cong}^-{\partial_U} \ar[d]_{f^*} & \bigoplus_{i=1}^n H^1(k_i,\BQ/\BZ)\ar[d]^{res}  \\
\Br_a(V)\ar[r]_-{(\partial_{D_{i,j}})}& \bigoplus_{i=1}^n \bigoplus_{j=1}^{g_i} H^1(D_{i,j},\BQ/\BZ)    } \end{equation}

Recall that $f$ is split if $D_{i, 1}$ is geometrically integral with multiplicity 1 in $D_i$ for all $1\leq i\leq n$. In this case, one has $\partial_{D_{i,1}}(b) \in H^1(k_i, \Bbb Q/\Bbb Z)$ for any $b\in \Br_1(V)$ with $1\leq i\leq n$.

\begin{prop} \label{fib-bm} Assume that $f$ is split.  Let $B$ be a finite subgroup of $\Br_1(V)$ and $F_i/k_i$ be a finite abelian extension such that  $\partial_{D_{i,1}}(B)|_{F_i} =0$ for all $1\leq i\leq n$.  Let $B_1$ be a finite subgroup of $\Br_1(V)$ such that the image of $B_1$ in $\Br_a(V)$ contains $$f^*(\partial_U^{-1}(\bigoplus_{i=1}^nH^1(F_i/k_i, \BQ/\BZ))) . $$ If $E=\prod_{v\in \infty_k} E_v$ is a connected component of $X(k_{\infty})$ such that $f(E_{v_0})$ is unbounded for some $v_0\in \infty_k$ when $k$ is totally real,  then 
$$ \bigcup_{c\in U(k)} X_c(\mathbf A_k)_{E_c}^{B+B_1} \ \ \ \text{is dense in} \ \ \ X(\mathbf A_k)_{E}^{ (B+B_1)\cap \Br_1(X)} $$ where $X_c$ is the fiber of $f$ over $c\in U(k)$ and $E_c=X_c(k_\infty) \cap E$. 
\end{prop}

\begin{proof} Write $\Bbb A_k^1=\Spec(k[t])$ and $P_i=(p_i(t))$ where $p_i(t)$'s are the fixed irreducible polynomials over $k$ for $1\leq i\leq n$. Let $X'$ be an open subset of $X$ over $k$ such that the restriction $f|_{X'}: X'\rightarrow \Bbb A_k^1$ is smooth and $f|_{X'}^{-1}(P_i)$ contains the smooth part $D_{i,1}^{sm}$ of $D_{i,1}$ for $1\leq i\leq n$.

 For any open subset $$W=E\times \prod_{v<\infty_k} W_v \subset X(\mathbf A_k) \ \ \ \text{ with } \ \ \ W^{(B+B_1)\cap \Br_1(X)}\neq \emptyset , $$ there is a finite subset $S$ of $\Omega_k$ containing $\infty_k$ such that the following conditions hold.

(a)  The morphism $f: X\rightarrow \Bbb A_k^1$ and the open immersion $X'\hookrightarrow X$ are extended to their integral models $f: \mathcal X\rightarrow \Bbb A_{\frak o_{k, S}}^1$ and  $\mathcal X'\hookrightarrow \mathcal X$ over $\frak o_{k, S}$ such that the restriction map $\mathcal X'\rightarrow \Bbb A_{\frak o_{k, S}}^1$ is smooth. 

(b) The morphism $D_i\rightarrow P_i$ is extended to their integral models $\mathcal D_i \rightarrow \mathcal P_i$ over $\frak o_{k,S}$ such that $$\mathcal P_i=\Spec(\frak o_{k,S}[t]/(p_i(t)))=\Spec (\frak o_{k_i, S} ) $$ is smooth over $\frak o_{k,S}$ and $\mathcal D_i=f^{-1}(\mathcal P_i)$ for $1\leq i\leq n$. Moreover, $$ \Bbb A_{\frak o_{k,S}}^1 \setminus \mathcal U=\{ \mathcal P_1, \cdots, \mathcal P_n\} \ \ \ \text{and} \ \ \ \mathcal{X} \setminus \mathcal V = \{ \mathcal D_1, \cdots, \mathcal D_n \}$$  

(c) The closed immersion $D_{i,1} \hookrightarrow D_i$ is extended to their models $\mathcal D_{i, 1}\hookrightarrow \mathcal D_i$ over $\frak o_{k_i, S}$ such that the smooth points over residue fields $\mathcal D^{sm}_{i, 1}(k(w))\neq \emptyset$ for all primes $w$ of $k_i$ with $(w\cap k)\not \in S$ for $1\leq i\leq n$. All field extensions $F_i/k_i/k$ are unramified outside $S$ for $1\leq i\leq n$.   

(d) All elements in $(B+B_1)\cap \Br_1(X)$ take the trivial value over $W_v=\mathcal X(\frak o_{k_v})$ and all elements in $B+B_1$ take the trivial value over $\mathcal V(\frak o_{k_v})$ for all $v\not\in S$.

If there is a complex prime $v\in \infty_k$, then $E_v=X(\Bbb C)$ by \cite[Chapter 3, Theorem 3.5]{PR} and $f(E_v)=\Bbb C$. In this case, we define $W_v= E_v\cap V(k_v)$ for all $v\in \infty_k$. Otherwise, $k$ is totally real. Then there is a connected component $C_{v_0}$ of $V(k_{v_0})$ such that $f(E_{v_0}\cap C_{v_0})$ is unbounded. In this case, we define
$$ W_v= \begin{cases} E_{v_0} \cap C_{v_0}  \ \ \ & \text{$v=v_0$} \\
E_v\cap V(k_v) \ \ \ & \text{$v\in \infty_k\setminus \{ v_0 \}$.} \end{cases} $$

By Harari's formal lemma (see \cite[Theorem 1.4]{CT01}), one can enlarge $S$ such that there are $x_v\in W_v$ for all $v\in S$  with 
\begin{equation} \label{shara}  \sum_{v\in S} \xi(x_v)=0  \end{equation} for all $\xi\in (B+B_1)$.  By shrinking $W_v$ for $v\in S\setminus \infty_k$, one can assume that each element in $B+B_1$ takes a single value over $W_v$. Since $f$ is split, the set $f(\prod_{v\in\Omega_k} W_v)$ contains a non-empty open set in $\mathbf A_k$ such that $f(W_{v_0})$ is unbounded when $k$ is totally real. By Proposition \ref{ref-sa}, there is $t_0\in \frak o_{k,S}$ such that 
$ X_{t_0}(\mathbf A_k) \cap W\neq \emptyset $. 

If $v\not\in S$ with $\ord_v(p_i(t_0))>0$ for some $1\leq i\leq n$, then $p_i(t)$ splits in $\frak o_{k_v}$ by (b) and Hensel's lemma. This implies that there is a prime $w$ of $k_i$ above $v$ such that $(k_i)_w=k_v$. There is $y_v\in \mathcal X_{t_0}'(\frak o_{k_v})\subset \mathcal X(\frak o_{k_v})$ such that 
$$ y_v \mod v\in \mathcal D^{sm}_{i,1}(k(v)) \subset  \mathcal X_{t_0}'(k(v)) $$
by Hensel's lemma and (a) and (c) for $1\leq i\leq n$. 
One can extend these $y_v$ to $$(y_v)_{v\in \Omega_k} \in X_{t_0}(\mathbf A_k)\cap W $$ by taking any element in $W_v\cap X_{t_0}(k_v)$ for $v\in S$ and any element in $\mathcal X_{t_0}(\frak o_{k_v})$ for the rest of primes $v$ of $k$. 
Then
$$\sum_{v\in \Omega_k} b(y_v) = \sum_{v\not\in S} b(y_v) = \sum_{i=1}^n \sum_{v\not\in S, \ \ord_v(p_i(t_0))>0} b(y_v) $$ by (\ref{shara}) for all $b\in (B+B_1)$. 

When $\ord_v(p_i(t_0))>0$ for some $1\leq i\leq n$ and $v\not\in S$, then $\ord_v(p_j(t_0))=0$ for all $j\neq i$ by (b). Therefore there is $\tau_i\in \Gal(F_i/k_i)$ such that $$ b(y_v) =\partial_{D_{i, 1}}(b) (\tau_i)  \ \ \ \text{with}  \ \ v\not\in S \ \ \text{and} \ \  \ord_v(p_i(t_0))>0$$ 
by \cite[Corollary 2.4.3]{Ha94} and the choice of $y_v$ for all $b\in B+B_1$. One concludes that  
\begin{equation} \label{s-res-bm}\sum_{v\in \Omega_k} b(y_v) =  \sum_{i=1}^n \partial_{D_{i,1}} (b) (\tau_i) \in \BQ/\BZ \end{equation}  
with $(\tau_i)_{i=1}^n \in \bigoplus_{i=1}^n \Gal(F_i/k_i)$ for all $b\in (B+B_1)$.

 Since 
$$ \sum_{v\in \Omega_k} f^*(\xi) (y_v) = \sum_{v\in \Omega_k} \xi (t_0) =0  $$  for all $\xi \in \partial_U^{-1} (\bigoplus_{i=1}^n H^1(F_i/k_i, \BQ/\BZ))$ by the functoriality of Brauer-Manin pairing and the reciprocity law, one obtains $$\chi ((\tau_i)_{i=1}^n)=0 \ \ \ \text{for all} \ \chi \in \bigoplus_{i=1}^n H^1(F_i/k_i, \BQ/\BZ)$$ by (\ref{residue}) and (\ref{s-res-bm}). 
This implies that $(\tau_i)_{i=1}^n$ is the trivial element in $\bigoplus_{i=1}^n \Gal(F_i/k_i)$. Therefore 
$$ (y_v)_{v\in \Omega_k} \in X_{t_0}(\mathbf A_k)^{B+B_1} \cap W$$ by (\ref{s-res-bm}) as desired. 
\end{proof}


\begin{rem}\label{sec} In the proof of Proposition \ref{fib-bm}, the assumption that $f$ is split is needed only for the fact that the set $f(\prod_{v\in\Omega_k} W_v)$ contains a non-empty open set in $\mathbf A_k$. When $f$ has a section over $k$, this fact is also true. Instead that $f$ is split, the same result still holds when $f$ has a section over $k$.
\end{rem}

Recall that a geometrically integral variety $Y$ over a field $k$ is quasi-trivial by \cite[Definition 1.1]{CT08} if the following two conditions hold:

(i)   $\bar k[Y]^\times/\bar k^\times$ is a permutation Galois module;

(ii) $\Pic(Y_{\bar k})=0 $.

 \begin{lem}\label{sansuc} Let $U$ be a quasi-trivial variety and $T$ be a torus over a number field $k$. 
If  $V \rightarrow U$ is a torsor under $T$, then the Sansuc morphism $\Br_a(V) \xrightarrow{\lambda}  \Br_a(T)$ by \cite[Lemma 6.4]{S} is surjective. 
 \end{lem}
 \begin{proof}  Since $V \rightarrow U$ is a torsor under $T$,  one gets an exact sequence 
 $$ 1 \rightarrow \bar{k}[U]^\times /\bar k^\times \rightarrow \bar{k} [V]^\times /\bar k^\times \rightarrow \bar k [T]^\times /\bar k^\times \rightarrow \Pic(U_{\bar k}) \rightarrow \Pic(V_{\bar k}) \rightarrow \Pic(T_{\bar k})  $$ by \cite[Proposition 6.10]{S}. Since $\Pic(U_{\bar k})= \Pic(T_{\bar k})=1$, one obtains $\Pic(V_{\bar k})=1$ and 
  \begin{equation} \label{coh-units} H^2(k, \bar{k}[U]^\times /\bar k^\times) \rightarrow H^2(k, \bar{k} [V]^\times /\bar k^\times ) \rightarrow H^2(k, \bar k [T]^\times /\bar k^\times) \rightarrow H^3(k, \bar{k}[U]^\times /\bar k^\times )  \end{equation}
by using Galois cohomology for the above short exact sequence.  Since $\bar{k}[U]^\times /\bar k^\times$ is a permutation $\Gal(\bar k/k)$-module, one obtains $H^3(k, \bar{k}[U]^\times /\bar k^\times )=0 $ by Shapiro lemma (see \cite[Chapter 1, (1.6.4) Proposition]{NSW}) and \cite[Chapter I, Corollary 4.17]{Mi06}. The Hochschild-Serre spectral sequence implies the canonical isomorphisms
$$ H^2(k, \bar{k} [V]^\times /\bar k^\times )\cong \Br_a(V) \ \ \ \text{and} \ \ \ H^2(k, \bar k [T]^\times /\bar k^\times)\cong \Br_a(T) .$$
The result follows from (\ref{coh-units}).
 \end{proof}

Now we further assume that $X$ admits an action of a torus $T$ over $k$. Namely, there is a morphism 
 $ T\times_k X \xrightarrow{\rho} X$ over $k$ satisfying the properties listed in \cite[Definition 0.3]{GIT}.

\begin{prop}\label{assum} Let $X$ be a smooth variety with an action of a torus $T$ over a number field $k$.  Suppose that $X\xrightarrow{f} Y$ is a morphism over  $k$ such that  $ f^{-1}(Z) \xrightarrow{f} Z$ is a torsor under $T$ where $Z$ is an open dense and quasi-trivial sub-variety of $Y$ over $k$. 

If $E$ is a connected component of $X(k_\infty)$ and $W=\prod_{v<\infty_k} W_v$ is an open compact subset of $X(\mathbf A_k^f)$, then there is a finite subgroup $B\subset \Br_a(f^{-1}(Z))$ such that 
$$ (E\times W)\cap X_c(\mathbf A_k)^B \neq \emptyset  \ \ \ \Longleftrightarrow \ \ \ (E\times W) \cap X_c(\mathbf A_k)^{\Br_a(X_c)} \neq \emptyset $$ for all $c\in Z(k)$. 
\end{prop}

\begin{proof} Let 
$$ \St(W_v)= \{ g\in T(k_v):  g \cdot W_v = W_v \}  $$ for $v<\infty_k$. 
Since $W_v$ is an open compact subset of $X(k_v)$ with $W_v= \mathcal X(\frak o_{k_v})$ for almost all $v$ for an integral model $\mathcal X$ of $X$, the group $\St(W_v)$ is an open subgroup of $T(k_v)$ such that $\St(W_v)=\mathcal T(\frak o_{k_v})$ for almost all $v$ for an integral model $\mathcal T$ of $T$.
Then 
$$ \St(W) =  T^+(k_\infty) \times \prod_{v<\infty_k} \St(W_v)  $$
is an open subgroup of $T(\mathbf A_k)$, where $T^+(k_\infty)$ is the connected Lie subgroup $T(k_{\infty_k})$. Let $$B_1=\{ \xi\in \Br_1(T): \ \sum_{v\in \Omega_k} \xi(\sigma_v)=0 \ \text{ for all $(\sigma_v)_{v\in \Omega_k}\in \St(W)$}\} . $$ 
 By \cite[Theorem 3.19]{D11}, one has the following commutative diagram of short exact sequences
\begin{equation}\label{finite-br} \xymatrix{ 0\ar[r] & T(\mathbf A_k)/\overline{T(k)\cdot T^+(k_\infty)}\ar[r]\ar@{->>}[d] & \Hom(\Br_a(T), \Bbb Q/\Bbb Z) \ar[r]\ar@{->>}[d] & \Sha^1(T) \ar[d]^{id} \ar[r] & 0  \\
0 \ar[r] & T(\mathbf A_k)/(T(k)\cdot \St(W)) \ar[r] & \Hom(B_1/\Br(k), \Bbb Q/\Bbb Z) \ar[r] & \Sha^1(T) \ar[r] & 0 } \end{equation}
and $B_1/\Br(k)$ is finite by \cite[Theorem 6.15 and Theorem 8.1]{PR}. 

Let $B$ be a finite subgroup of $\Br_a(f^{-1}(Z))$ such that  $\lambda (B)=B_1/\Br(k)$ where $\lambda$ is the Sansuc morphism in Lemma \ref{sansuc}.  If $$ (E\times W)\cap X_c(\mathbf A_k)^B \neq \emptyset $$ for $c\in Z(k)$,  there is $(x_v)_{v\in \Omega_k} \in (E\times W)\cap X_c(\mathbf A_k)$ which can be viewed  $$(x_v)_{v\in \Omega_k} \in \Hom(\Br_a(X_c), \Bbb Q/\Bbb Z) \ \ \ \text{such that} \ \ \  (x_v)_{v\in \Omega_k} |_B=0 .$$ By \cite[Lemma 6.8]{S}, there is a canonical isomorphism $\lambda_c: \Br_a(X_c)\xrightarrow{\cong} \Br_a(T)$ such that the following diagram commutes 
\begin{equation}\label{sansuc-cd} \xymatrix{ \Br_a(f^{-1}(Z)) \ar[d]_{\lambda} \ar[r]^{i_c^*} & \Br_a(X_c)\ar[ld]^{\lambda_c}  \\
 \Br_a(T)} \end{equation} where $i_c: X_c\hookrightarrow f^{-1}(Z)$ is the closed immersion. Since $$(x_v)_{v\in \Omega_k} \circ \lambda_c^{-1} \in  \Hom (\Br_a(T), \Bbb Q/\Bbb Z)$$ with $$ (x_v)_{v\in \Omega_k} \circ \lambda_c^{-1}(B_1/\Br(k))=(x_v)_{v\in \Omega_k} \circ \lambda_c^{-1}\circ \lambda (B)=i_c^*(B)((x_v)_{v\in \Omega_k})=0 $$ by (\ref{sansuc-cd}), 
 there are $t\in T(k)$ and $(s_v)_{v\in \Omega_k} \in \St(W)$ such that 
 \begin{equation}\label{modif}  t\cdot(s_v)_{v\in \Omega_k}=  (x_v)_{v\in \Omega_k} \circ \lambda_c^{-1} \in  \Hom (\Br_a(T), \Bbb Q/\Bbb Z) \end{equation} by (\ref{finite-br}). Consider $$(s_v^{-1}\cdot x_v)_{v\in \Omega_k} \in (E\times W)\cap X_c(\mathbf A_k) . $$ Then
$$ b((s_v^{-1} \cdot x_v)_{v\in \Omega_k})= \sum_{v\in \Omega_k} \lambda_c(b)(s_v^{-1}) + \sum_{v\in \Omega_k} b(x_v) =-\sum_{v\in \Omega_k} \lambda_c(b)(t\cdot s_v)+ \sum_{v\in \Omega_k} b(x_v) =0 $$ for all $b\in \Br_a(X_c)$ by (\ref{modif}) and \cite[Proposition 2.9]{CDX}. Namely, 
$$ (s_v^{-1}\cdot x_v)_{v\in \Omega_k} \in (E\times W) \cap X_c(\mathbf A_k)^{\Br_a(X_c)} \neq \emptyset $$ as desired. 
\end{proof}

Combining Proposition \ref{fib-bm} and Proposition \ref{assum}, one obtains the main result of this section.

\begin{thm} \label{main-1} Let $X\xrightarrow{f} \BA_k^1$ be a surjective morphism over a number field $k$. Assume that $X$ admits an action of a torus $T$ over $k$ such that the generic fiber of $f$ is a torsor under $T$. Suppose that each equivalent class of admissible connected components of $X(k_\infty)$ contains a connected component $E=\prod_{v\in \infty_k} E_v$ such that $f(E_v)$ is unbounded for some $v\in \infty_k$ when $k$ is totally real. If $f$ is split or has a section over $k$, then $X$ satisfies strong approximation off $\infty_k$ with respect to $\Br_1(X)$. 
\end{thm}

\begin{proof} For any open subset $X(k_\infty)\times \prod_{v<\infty_k} W_v$ of $X(\mathbf A_k)$ with $$(X(k_\infty)\times \prod_{v<\infty_k} W_v)\cap X(\mathbf A_k)^{\Br_1(X) }\neq \emptyset, $$ there is a connected component $E$ of $X(k_\infty)$ such that 
$$ (E\times \prod_{v<\infty_k} W_v) \cap X(\mathbf A_k)^{\Br_1(X) }\neq \emptyset $$ and 
$f(E)$ is unbounded.  Since $X$ admits an action of a torus $T$ over $k$ such that the generic fiber of $f$ is a torsor under $T$, there is an open dense subset $U$ of $\BA_k^1$ over $k$ with $V=f^{-1}(U)$ such that $f|_V: V \rightarrow U$ is a torsor under $T$. By Proposition \ref{assum}, there is a finite subgroup $B\subset \Br_1(V)$ such that 
$$ (E\times \prod_{v<\infty_k} W_v)\cap X_c(\mathbf A_k)^B \neq \emptyset  \ \ \ \Longleftrightarrow \ \ \ (E\times \prod_{v<\infty_k} W_v) \cap X_c(\mathbf A_k)^{\Br_a(X_c)} \neq \emptyset $$ for all $c\in U(k)$.  Let $B_1$ be a finite subgroup of $\Br_1(V)$ as defined in Proposition \ref{fib-bm} for $B$. Then there is $c_0\in U(k)$ such that $$(E\times \prod_{v<\infty_k} W_v)\cap X_{c_0} (\mathbf A_k)^{B+B_1} \neq \emptyset $$ by Proposition \ref{fib-bm} and Remark \ref{sec}. Therefore 
$$(E\times \prod_{v<\infty_k} W_v) \cap X_{c_0}(\mathbf A_k)^{\Br_a(X_{c_0})} \neq \emptyset . $$ This implies that $X_{c_0}$ is a trivial torsor by \cite[Corollary 8.7]{S} and 
$$ X_{c_0}(k) \cap (E\times W) \neq \emptyset $$ by \cite[Theorem 3.19]{D11}. 
\end{proof}

\begin{cor} \label{app-neq} Let $X$ be the smooth locus of the following affine variety  
$$  \prod_{i=1}^m N_{L_i/k}(x_i) = c\prod_{j=1}^n p_j(t)^{e_j} \ \ \ \text{with} \ \ \ c\in k^\times $$ 
where $L_i/k$'s are finite extensions of number fields and $p_j(t)$'s are distinct irreducible monic polynomials over $k$ and $e_j$'s are positive integers. 
If for each $1\leq j \leq n$ there is $L_i$ such that  either $L_i$ can be imbedded in $k[t]/(p_j(t))$ over $k$ or $[L_i: k] \mid e_j$ with $1\leq i\leq m$, then $X$ satisfies strong approximation off $\infty_k$ with respect to $\Br_1(X)$.  
\end{cor}
\begin{proof}  Consider the following homomorphism of tori
$$\phi: \  \prod_{i=1}^m \Res_{L_i/k} (\Bbb G_m) \rightarrow \Bbb G_m ; \ \ (x_1, \cdots, x_m) \mapsto \prod_{i=1}^m N_{L_i/k} (x_i) $$
and $T=\ker \phi$. Then the surjective morphism
$$ f:  X\rightarrow \Bbb A_k^1;  \ \ (x,t) \mapsto t $$  with the action of $T$ $$T \times X\rightarrow X; \ (\alpha, (x,t))\mapsto (\alpha \cdot x, t) . $$ 
For each $1\leq j \leq n$, there is $L_i$ such that  either $L_i$ can be imbedded in $k[t]/(p_j(t))$ over $k$ or $[L_i: k] \mid e_j$ with $1\leq i\leq m$. This implies that each $p_j(t)^{e_j}$ 
can be written as a norm from $L_i$ to $k$ with the variable $t$ for $1\leq j\leq n$. Since $$X(\mathbf A_k)^{\Br_1(X)}\neq \emptyset \ \ \ \Rightarrow \ \ \ c\in N_{L/k}(L^\times) $$ by \cite[Theorem 5.1]{S}, one concludes that $f$ has a section over $k$. 

If there are a real prime $v\in \infty_k$ and $L_i$ with $1\leq i\leq m$ such that $L_i$ has a real prime above $v$, then $X(k_v)$ is connected by inspecting the equation of $X$. Therefore one only needs to consider that $k$ is totally real and all primes of $L_i$ above $\infty_k$ are complex for all $1\leq i\leq m$. In this case, one can list all real roots of $p_i(t)^{e_i}$ with the odd integers $e_i$ 
$$-\infty< \alpha_1<\alpha_2 < \cdots < \alpha_s< \infty .$$ Since the sign of  $\prod_{i=1}^n p_i(t)^{e_i}$ will change when $t$ crosses such a root, one can concludes that $$f(E_v)\cap f(E_v')=\emptyset$$ where $E_v$ and $E_v'$ are two different connected components of $X(k_v)$ for $v\in \infty_k$. On the other hand, the morphism $f$ has a section. This implies that $X(k_v)$ is connected and $f(X(k_v))= k_v$. 
\end{proof}



\begin{exam} Let $a\in k^\times\setminus (k^\times)^2$. For any positive integer $d$, there is an irreducible polynomial $f(t)$ over $k$ with $deg(f)= 2d$ such that the affine Ch\^{a}telet surface
$$  x^2-a y^2 =f(t)$$ satisfies strong approximation off $\infty_k$. 
\end{exam}
\begin{proof}  Let $L=k(\sqrt{a})$ and $\Bbb Q^{ab}$ be the maximal abelian extension of $\Bbb Q$. Since $$\Gal(\Bbb Q^{ab}/\Bbb Q)=\prod_{p \ \text{prime}} \Bbb Z_p^\times , $$
one concludes that $\Gal(k\cdot \Bbb Q^{ab}/k)$ is an open subgroup of $\Gal(\Bbb Q^{ab}/\Bbb Q)$ with a finite index. This implies that there is a Galois extension $K/k$ containing $L$ with $[K:k]=2d$. Let $K=k(\theta)$ and $f(x)$ be an irreducible monic polynomial of $\theta$ over $k$. Write $X$ a variety defined by the equation $$ x^2-a y^2 =f(t) . $$
Since $\bar{k}[X]^\times=\bar k^\times$  by \cite[Proposition 2]{DSW},  one has
$$\Br_a(X)= H^1(k, \Pic(X_{\bar k}))$$ by \cite[Lemma 2.1]{CTX09} and  \cite[Chapter 1, Corollary 4.21]{Mi06}. Let $U$ be an open subset of $X$ defined by $f(t)\neq 0$. Then
 $\Pic(U_{\bar k})=1$. Moreover, one has the short exact sequence
$$ 1\rightarrow \bar k [U]^\times/\bar k^\times \xrightarrow{div} \Div_{X_{\bar k}\setminus U_{\bar k}}(X_{\bar k}) \rightarrow \Pic(X_{\bar k}) \rightarrow 1 .$$
This implies that $\Pic(X_{\bar k})$ is a free abelian group of rank $2d-1$. By the inflation-restriction sequence,  one further has $$\Br_a(X)= H^1(k, \Pic(X_{\bar k}))= H^1(K/k, \Pic(X_K)) . $$
Since $\Pic(U_K)=1$, one has the short exact sequence 
$$ 1\rightarrow K[U]^\times/K^\times \rightarrow \Div_{X_K\setminus U_K}(X_K) \rightarrow \Pic(X_K) \rightarrow 1 .$$ This implies the exact sequence 
$$ H^1(K/k, \Div_{X_K\setminus U_K}(X_K)) \rightarrow H^1(K/k, \Pic(X_K) )\rightarrow H^2(K/k,  K[U]^\times/K^\times) $$ by Galois cohomology. Since $\Div_{X_K\setminus U_K}(X_K)$ is a permutation $\Gal(K/k)$-module, one has $$H^1(K/k, \Div_{X_K\setminus U_K}(X_K))=0$$ by Shapiro's lemma (see \cite[(1.6.3) Proposition]{NSW}). Since 
$$K[U]^\times/K^\times\cong \Bbb Z (x-\sqrt{a} y) \oplus (\bigoplus_{\sigma\in \Gal(K/k)} \Bbb Z(t-\sigma\theta)) ,$$
one concludes $H^1(K/L, K[U]^\times/K^\times)= 0$. Therefore one obtains the following the inflation-restriction sequence
$$ 1\rightarrow H^2(L/k, (K[U]^\times/K^\times)^{\Gal(K/L)}) \rightarrow H^2(K/k,  K[U]^\times/K^\times)  \rightarrow H^2(K/L,  K[U]^\times/K^\times)^{\Gal(L/k)} .$$
Since $$H^2(K/L,  K[U]^\times/K^\times)=H^2(K/L,  \Bbb Z (x-\sqrt{a} y))=H^1(K/L, \Bbb Q/\Bbb Z (x-\sqrt{a}y)) $$ by Shapiro's lemma, one has 
$$H^2(K/L,  K[U]^\times/K^\times)^{\Gal(L/k)} = \Hom(\Gal(K/L),\Bbb Q/\Bbb Z (x-\sqrt{a}y))^{\Gal(L/k)} =0 .$$
Note
$$(K[U]^\times/K^\times)^{\Gal(K/L)}\cong \Bbb Z (x-\sqrt{a} y) \oplus \Bbb Z f_1 (x) \oplus \Bbb Z f_2 (x) $$ where $$f_1(t) =\prod_{\sigma\in \Gal(K/L)} (t-\sigma \theta) \ \ \ \text{ and } \ \ \ f_2(t) =\prod_{\sigma\in \Gal(K/k)\setminus \Gal(K/L)} (t-\sigma \theta) $$ satisfying $\tau f_1(t)=f_2(t)$ for the non-trivial element $\tau\in \Gal(L/k)$. Since 
$$\tau (x-\sqrt{a} y)= f_1(t)\cdot f_2(t) \cdot (x-\sqrt{a} y)^{-1} , $$
one concludes that 
$$ H^2(L/k, (K[U]^\times/K^\times)^{\Gal(K/L)}) =\hat{H}^0(L/k, (K[U]^\times/K^\times)^{\Gal(K/L)}) =0 .$$ Therefore $\Br_a(X)=0$. Since $f(t)$ is a monic polynomial, each connected component of $X(k_\infty)$ is not compact. The result follows from  Corollary \ref{app-neq}. 
\end{proof}

The local-global principle for integral points of Diophantine equations in \cite{Wat} can be reduced to study strong approximation of 
\begin{equation}\label{watson}  q(x_1, \cdots, x_n) = p(t)  \end{equation} where $q(x_1, \cdots, x_n)$ is a non-degenerated quadratic form and $p(t)$ is a non-zero polynomial over a number field $k$. 
Indeed, strong approximation has been established in \cite{CTX13} and \cite{Xu} for $n\geq 3$. The remaining $n=2$ case is an affine Ch\^{a}telet surface and \cite[Theorem 4]{Wat} can be recovered from the following example when the equation (\ref{watson}) is smooth. 

\begin{exam} Let $X$ be the smooth variety defined by (\ref{watson}) with $n=2$ over $k$. If $q(x_1, x_2)=0$ has a non-trivial solution over $k$ and $p(t)\equiv 0 \mod v$ is soluble for almost all primes $v$ of $k$, then $X$ satisfies strong approximation off $\infty_k$. 
\end{exam}
\begin{proof} Since $q(x_1, x_2)=0$ has a non-trivial solution over $k$, one can assume that $$q(x_1, x_2)=x_1x_2 . $$ 
Applying Corollary \ref{app-neq} for $m=2$ and $L_1=L_2=k$, one concludes that $X$ satisfies strong approximation off $\infty_k$ with respect to $\Br_1(X)$. 

Let $U$ be an open subset of $X$ defined by $x_1\neq 0$. Then $U\cong \Bbb G_m \times_k \Bbb A_k^1$ over $k$. This implies that $\bar k[U]^\times/\bar k^\times$ is free of rank 1 with a generator $x_1$. Since $p(t)\equiv 0 \mod v$ is soluble for almost all primes $v$ of $k$, then $p(t)$ is not a constant. Therefore $x_1\not \in \bar k[X]^\times$ and  $\bar k[X]^\times = \bar k^\times$.  One concludes that 
\begin{equation} \label{bp} \Br_a(X)= H^1(k, \Pic(X_{\bar k})) \end{equation} by \cite[Lemma 2.1]{CTX09} and  \cite[Chapter 1, Corollary 4.21]{Mi06} and the short exact sequence
\begin{equation}\label{ses} 1\rightarrow \bar k [U]^\times/\bar k^\times \xrightarrow{div} \Div_{X_{\bar k}\setminus U_{\bar k}}(X_{\bar k}) \rightarrow \Pic(X_{\bar k}) \rightarrow 1 \end{equation} of $\Gal(\bar k/k)$-modules. 
Since $X$ is smooth, one can write $$ p(t)=c\prod_{j=1}^n p_j(t) \ \ \ \text{and} \ \ \ K_j= k[t]/(p_j(t))$$ with $c\in k^\times$ and $1\leq j\leq n$,  
where $p_j(t)$'s are distinct irreducible monic polynomials over $k$. Then 
$$\Div_{X_{\bar k}\setminus U_{\bar k}}(X_{\bar k})\cong \bigoplus_{j=1}^n \Ind_{\Gal(\bar k/k)}^{\Gal(\bar k/K_j)} \Bbb Z $$ as $\Gal(\bar k/k)$-modules.
Since  $$H^1(k, \Div_{X_{\bar k}\setminus U_{\bar k}}(X_{\bar k}))=0 \ \ \ \text{and} \ \ \ H^2(k, \Div_{X_{\bar k}\setminus U_{\bar k}}(X_{\bar k}))=\bigoplus_{j=1}^n H^2(K_j, \Bbb Z)  $$  by Shapiro's lemma (see \cite[(1.6.3) Proposition]{NSW}) and 
$$ 0=H^1(k, \Div_{X_{\bar k}\setminus U_{\bar k}}(X_{\bar k})) \rightarrow H^1(k,  \Pic(X_{\bar k})) \rightarrow H^2(k, \bar k [U]^\times/\bar k^\times) \rightarrow H^2(k, \Div_{X_{\bar k}\setminus U_{\bar k}}(X_{\bar k}))  $$ by applying Galois cohomology to (\ref{ses}), one obtains that 
$$ \Br_a(X) \cong \ker (H^1(k, \Bbb Q/\Bbb Z) \xrightarrow{res} \bigoplus_{j=1}^n H^1(K_j, \Bbb Q/\Bbb Z)) $$ by (\ref{bp}) and the identifications $$H^2(k, \Bbb Z)\cong H^1(k, \Bbb Q/\Bbb Z) \ \ \ \text{ and } \ \ \ H^2(K_j, \Bbb Z)\cong H^1(K_j, \Bbb Q/\Bbb Z)$$ with $1\leq j\leq n$. 

Suppose $\Br_a(X)$ is not trivial. There is 
\begin{equation} \label{kh} \psi \in \ker [\Hom_{cts}(\Gal(\bar k/k), \Bbb Q/\Bbb Z) \xrightarrow{res} \bigoplus_{j=1}^n \Hom_{cts}(\Gal(\bar k/K_j), \Bbb Q/\Bbb Z)] \end{equation} such that $\ker(\psi)= \Gal(\bar k/L)$ where $L/k$ is a finite abelian extension with $\sigma\in \Gal(L/k)$ and $\psi(\sigma)\neq 0$. By the Chebotarev density theorem, there are infinitely many primes $\frak p$ of $k$ such that $Frob_{L/k}(\frak p)=\sigma$. Since $p(t)\equiv 0 \mod v$ is soluble for almost all primes $v$ of $k$, there is a prime $v_0$ of $k$ such that $$Frob_{L/k}(v_0)=\sigma \ \ \ \text{ and } \ \ \ (K_j)_{w}= k_{v_0}$$ for some $1\leq j\leq n$ and some prime $w$ of $K_j$ above $v_0$. This implies that $$Frob_{LK_j/K_j}(w)|_{L/k} = \sigma  \ \ \ \text{and} \ \ \ \psi(\sigma)=0 $$ by (\ref{kh}), which contradicts the choice of $\psi$. Therefore $\Br_a(X)$ is trivial and $X$ satisfies strong approximation off $\infty_k$. 
\end{proof}

It should be pointed out that the Brauer groups of certain affine Ch\^{a}telet surfaces are computed with some algorithm for Brauer-Manin set in \cite{berg}.

\section{Application of Schinzel hypothesis to non-split fibrations}

We keep the same notations as those in the previous section if not explained.  More precisely,  we assume that  $X$ is a smooth and geometrically integral variety  over a number field $k$. Let $X\xrightarrow{f} \Bbb A_k^1$ be a surjective morphism over $k$ with geometrically integral generic fiber. The significant difference is that we do not assume that $f$ is split in this section. Then the map on the adelic points $X(\mathbf A_k) \rightarrow \mathbf A_k$ induced by $f$ is not necessarily open. This prevents us from applying strong approximation for $\Bbb A_k^1$ to find a rational point over $\Bbb A_k^1$. 

\begin{exam} Let $X$ be a smooth and geometrically integral variety defined by $x^2+y^2=t$ over $\Bbb Q$ and consider the fibration $X\xrightarrow{f} \Bbb A_{\Bbb Q}^1$ by sending $(x,y,t)\mapsto t$. For the prime $p\equiv 3 \mod 4$ and $\ord_p(t) \equiv 1\mod 2$, the equation $x^2+y^2=t$ has no solution over $\Bbb Z_p$.  There are infinitely many primes $p$ such that the map $f: \mathcal X(\Bbb Z_p) \rightarrow \Bbb Z_p$ is not surjective. Therefore the map $X(\mathbf A_\Bbb Q) \rightarrow \mathbf A_\Bbb Q$ induced by $f$ is not open.
\end{exam}

To overcome (D2) in \S 1, one can apply Schinzel's hypothesis. The hypothesis ($H_1$) in \cite[pp. 71]{CTSD}  is a consequence of Schinzel's hypothesis 
which is due to Serre by \cite[Prop. 4.1]{CTSD}. One can further refine the hypothesis ($H_1$) to the following hypothesis with the signs.  

\begin{hypo}[$H_s$] \label{schinzel} Let $p_1(t), \cdots, p_n(t)$ be irreducible polynomials over a number field $k$ and $\epsilon_v\in \{\pm 1\}$ for each real prime $v\in \infty_k$. Then there is a finite subset $S$ of $\Omega_k$ containing $\infty_k$ satisfying the following property:

For any $\lambda_v \in k_v$ with $v\in S\setminus \infty_k$, there is $\lambda\in \frak o_{k,S}$ close to $\lambda_v$ for $v\in S\setminus \infty_k$ with arbitrarily large $\epsilon_v \lambda$ at all real primes $v$ of $k$ such that for each $1\leq i\leq n$, $p_i(\lambda)$ is a unit of $k_w$ for all primes $w\not\in S$ except one prime $w_i$, where $p_i(\lambda)$ is a uniformizing parameter.




\end{hypo}

\begin{prop}\label{hs} Schinzel hypothesis over $\Bbb Q$ implies ($H_s$) over a number field $k$. 

\end{prop}
\begin{proof} Fix $a\in k$ such that $\epsilon_v a>0$ for all real primes $v\in \infty_k$ and write $P_i(t)=p_i(at)$ with $1\leq i\leq n$. Let $S$ be a finite set of primes of $k$ satisfying \cite[\bf{Hypothesis ($H_1$)}]{CTSD} for $P_1(t), \cdots, P_n(t)$ and $a\in \frak o_{k,S}$. For any given $\lambda_v \in k_v$ with $v\in S\setminus \infty_k$, there is $t_0\in \frak o_{k, S}$ close to $a^{-1}\lambda_v$ for $v\in S\setminus \infty_k$ and arbitrarily large at all real primes $v\in \infty_k$ such that for each $1\leq i\leq n$, $P_i(t_0)$ is a unit of $k_w$ for all primes $w\not\in S$ except one prime $w_i$, where $P_i(t_0)$ is a uniformizing parameter by \cite[Prop.4.1]{CTSD}. Then one can choose $\lambda = at_0$ as desired. 
\end{proof}

Let $U$ be an open dense subset of $\Bbb A_k^1$ over $k$ such that $f|_V: V=f^{-1}(U) \rightarrow U$ is smooth with geometrically integral fibers.  Write $$\Bbb A_k^1\setminus U=\{P_1, \cdots, P_n \}$$ where $P_1, \cdots, P_n$ are the closed points over $k$ and $k_i=k(P_i)$ are the residue fields of $P_i$ for $1\leq i\leq n$. Let $D_i=f^{-1}(P_i)$ and $\{D_{i,j}\}_{j=1}^{g_i}$ be the set of irreducible components of $D_i$ for $1\leq i\leq n$. For any  $b\in \Br_1 (V)$,  one has $$\partial_{D_{i,j}}(b) \in H^1(L_{i, j}, \Bbb Q/\Bbb Z)$$ where  $L_{i, j}$ is the algebraic closure of $k_i$ inside the function field $k_i(D_{i, j})$ for $1\leq j\leq g_i$ with $1\leq i\leq n$.

The following result is an analogue of Proposition \ref{fib-bm} without assuming that all fibers of $f$ are split under Schinzel's hypothesis. 

\begin{prop} \label{sc-bm}  Let $B$ is a finite subgroup of $\Br_1(V)$. Suppose that $E=\prod_{v\in \infty_k} E_v$ is a connected component of $X(k_\infty)$ such that $f(E_v)$'s are unbounded for all real primes $v$ of $k$. Assume that

(i)  Each $D_{i}$ contains an irreducible component $D_{i,1}$ of multiplicity 1 such that the algebraic closure $L_i$ of $k_i$ inside $k_i(D_{i,1})$ is cyclic over $k_i$ with $1\leq i\leq n$.

(ii)  Hypothesis ($H_s$) holds. 

Let $M_i/L_{i}$ be a finite abelian extension such that  $\partial_{D_{i,1}}(B)|_{M_i} =0$ and $M_i/k_i$ is Galois for $1\leq i\leq n$. Let $B_1$ be a finite subgroup of $\Br_1(V)$ such that the image of $B_1$ in $\Br_a(V)$ contains $$f^*(\partial_U^{-1}(\bigoplus_{i=1}^nH^1(M_i/k_i, \BQ/\BZ)))$$ by the commutative diagram (\ref{residue})

Then 
$$ \bigcup_{c\in U(k)} X_c(\mathbf A_k)_{E_c}^{B+B_1} \ \ \ \text{is dense in} \ \ \ X(\mathbf A_k)_E^{(B+B_1)\cap \Br_1(X)} $$ where $X_c$ is the fiber of $f$ over $c\in U(k)$ and $E_c=E\cap X_c(k_\infty)$. 
\end{prop}

\begin{proof} Write $\Bbb A_k^1=\Spec(k[t])$ and $P_i=(p_i(t))$ where $p_i(t)$'s are the fixed irreducible polynomials over $k$. Let $X'$ be an open subset of $X$ over $k$ such that the restriction $f|_{X'}: X'\rightarrow \Bbb A_k^1$ is smooth and $f|_{X'}^{-1}(P_i)$ contains the smooth part $D_{i,1}^{sm}$ of $D_{i,1}$ for $1\leq i\leq n$. 
 
For any open subset $$W=E\times \prod_{v<\infty_k} W_v \subset X(\mathbf A_k) \ \ \ \text{ with } \ \ \ W^{\Br_1(X) }\neq \emptyset , $$ there is a finite subset $S_0$ of $\Omega_k$ containing $\infty_k$ such that the following conditions hold.

(a)  The morphism $f: X\rightarrow \Bbb A_k^1$, the open immersions $U\hookrightarrow \Bbb A_k^1$,  $V\hookrightarrow X$ and $X'\hookrightarrow X$ are extended to their integral models over $\frak o_{k, S_0}$ with the following commutative diagram 
$$ \xymatrix{ \mathcal{V} \ar[d]_f  \ar@{^{(}->}[r]  & \mathcal{X} \ar[d]^{f}  \\
\mathcal U\ar@{^{(}->}[r] &  \Bbb A_{\frak o_{k,S_0}}^1 }   $$
such that $f(\mathcal V(\frak o_{k_v})) = \mathcal U(\frak o_{k_v}) $ for $v\not\in S_0$ and the restriction $f|_{\mathcal X'}: \mathcal X' \rightarrow  \Bbb A_{\frak o_{k,S_0}}^1$ is smooth.

(b) The morphism $D_i\rightarrow P_i$ is extended to their integral models $\mathcal D_i \rightarrow \mathcal P_i$ over $\frak o_{k,S_0}$ such that $$\mathcal P_i=\Spec(\frak o_{k,S_0}[t]/(p_i(t)))=\Spec (\frak o_{k_i, S_0} ) $$ is smooth over $\frak o_{k,S_0}$ and $\mathcal D_i=f^{-1}(\mathcal P_i)$ for $1\leq i\leq n$. Moreover, $$ \Bbb A_{\frak o_{k,S_0}}^1 \setminus \mathcal U=\{ \mathcal P_1, \cdots, \mathcal P_n\} \ \ \ \text{and} \ \ \ \mathcal{X} \setminus \mathcal V = \{ \mathcal D_1, \cdots, \mathcal D_n \}$$  

(c) The closed immersion $D_{i,1} \hookrightarrow D_i$ is extended to their models $\mathcal D_{i, 1}\hookrightarrow \mathcal D_i$ over $\frak o_{k, S_0}$ 
such that the smooth points over residue fields $\mathcal D^{sm}_{i, 1}(k(w))\neq \emptyset$ for all primes $w$ in $L_i$ with $(w\cap k)\not \in S_0$ for $1\leq i\leq n$.
All field extensions $M_i/L_i/k_i/k$ are unramified outside $S_0$ for $1\leq i\leq n$.   

(d) All elements in $(B+B_1 )\cap \Br(X)$ take the trivial value over $W_v=\mathcal X(\frak o_{k_v})$ and all elements in $B+B_1$ take the trivial value over $\mathcal V(\frak o_{k_v})$ for all $v\not\in S_0$. 
 
 By the Chebotarev density theorem, there is a finite set $T_i$ of primes of degree 1 over $k$ in $L_i$ with $|\Gal(M_i/k_i)|$ elements such that 

(1)   for each $\sigma \in \Gal(M_i/L_i)$, there are exactly $|\Gal(L_i/k_i)|$ primes $\frak q\in T_i$ with sufficiently large cardinalities of the residue fields and $Frob(\frak q) =\sigma$ for $1\leq i\leq n$;

(2)  the set $S_i =\{ \frak q \cap k: \ \frak q \in T_i\}$ has $|\Gal(M_i/k_i)|$ primes with $S\cap S_i=\emptyset$ for $1\leq i\leq n$ and $S_i \cap S_j=\emptyset $ for $i\neq j$.

 Since all primes in $T_i$ are of degree 1 over $k$,  one concludes that $(L_i)_\frak q=k_\frak p$ for each $\frak q\in T_i$ with $\frak p=\frak q\cap k\in S_i$ for $1\leq i\leq n$. Moreover, the scheme $\mathcal D_{i,1} \times_{\frak o_{k,S}} \frak o_{k_{\frak p}}$ splits into 
  geometrically irreducible components $\{ \sigma\mathcal D_{i, \frak p}: \sigma\in \Gal(L_i/k_i)\} $ over $\frak o_{k_{\frak p}}$ by (b) and (c), where $\mathcal D_{i, \frak p}$ is a fixed geometrically irreducible component of $\mathcal D_{i,1} \times_{\frak o_{k,S}} \frak o_{k_{\frak p}}$ for $\frak p\in S_i$ with $1\leq i\leq n$. 
 Since the polynomial $p_i(t)$ has a root over $\frak o_{k_{\frak p}}$ for each $\frak p\in S_i$ by (b) and (c), one obtains $t_{\frak p}\in \frak o_{k_{\frak p}}$ with $\ord_{\frak p}(p_i(t_{\frak p}))=1$ such that there is $x_\frak p\in \mathcal X'_{t_\frak p}(\frak o_{k_\frak p})\subset \mathcal X(\frak o_{k_\frak p})$ satisfying   $$ x_\frak p \mod \frak p \in   \mathcal D_{i, \frak p}^{sm}(k(\frak p)) \subset \mathcal X'_{t_\frak p}(k(\frak p)) $$ by (1), the Lang-Weil's estimation (see \cite[Theorem 7.7.1]{Po}), (a) and Hensel's lemma for $1\leq i\leq n$. 

For any real prime $v\in \infty_k$, there is a connected component $C_{v}$ of $V(k_{v})$ such that $f(E_{v}\cap C_{v})$ is unbounded.
We define
$$ W_v= \begin{cases} E_{v} \cap C_{v}  \ \ \ & \text{$v$ real prime in $\infty_k$} \\
E_v\cap V(k_v) \ \ \ & \text{$v$ complex prime in $\infty_k$.} \end{cases} $$

By Harari's formal lemma (see \cite[Theorem 1.4]{CT01}), there is a finite set $S$ of primes of $k$ containing $\cup_{i=0}^n S_i$  with
$$ \begin{cases}  x_\frak p  \in \mathcal X'_{t_\frak p}(\frak o_{k_\frak p}) \ \ \ & \text{as above for} \  \frak p \in \cup_{i=1}^n S_i \\
x_v \in W_v  \  \ \ & \text{for} \  v\in S\setminus(\cup_{i=1}^n S_i) \end{cases} $$
such that 
\begin{equation}\label{hara}  \sum_{v\in S} \xi(x_v)=0   \end{equation} 
for all $\xi\in (B+B_1)$. By shrinking $W_v$ for $v\in S\setminus \infty_k$, one can assume that each element in $B+B_1$ takes a single value over $W_v$ and $f(W_v)$ is open in $k_v$.

By (ii), there is $t_0\in \frak o_{k, S}$ close to $f(x_v)$ for $v\in S\setminus \infty_k$ and $t_0\in f(W_v)$ for $v\in \infty_k$ such that $\ord_{v}(p_i(t_0))=0$ for all $v\not\in S$ except one prime $v_i$ of $k$ with $\ord_{v_i}(p_i(t_0))=1$ for $1\leq i\leq n$. Write $$p_i(t)=\prod_{j=1}^g q_j(t)$$ where $q_j(t)$ are irreducible polynomials over $\frak o_{k_{v_i}}$ for $1\leq j\leq g$. There are $g$ primes $w_1, \cdots, w_g$ of $k_i$ above $v_i$ corresponding to irreducible polynomials $q_1(t), \cdots, q_g(t)$ respectively. Since $\ord_{v_i}(p_i(t_0))=1$, there is $1\leq j_0\leq g$ such that 
\begin{equation} \label{spr} \ord_{v_i}(q_{j_0}(t_0))=1 \ \ \ \text{and} \ \ \ \ord_{v_i}(q_j(t_0))=0 \ \ \ \text{for all $j\neq j_0$.} \end{equation}  
Then $q_{j_0}(t)$ is of degree 1 by (b) and Hensel's lemma and $w_{j_0}$ is a prime of degree 1 over $v_i$.

Let $\alpha_i$ be the image of $t$ in $k[t]/(p_i(t))=k_i$. Then $$Cores_{k_i/k} (\chi, t-\alpha_i) \in \Br_1(U)  \  \ \text{with} \ \ f^*(Cores_{k_i/k} (\chi, t-\alpha_i))\in B_1$$ for any $\chi\in H^1(L_i/k_i, \Bbb Q/\Bbb Z)$ by \cite[\S1.2]{CTSD}. Since  
$$ 0= \sum_{v\in \Omega_k} Cores_{k_i/k} (\chi, t_0-\alpha_i)_v = Cores_{k_i/k} (\chi,  t_0-\alpha_i )_{v_i} +\sum_{v\in S} Cores_{k_i/k} (\chi, t_0-\alpha_i)_v $$ by the reciprocity law, (b) and (c),  one concludes that
$$\sum_{j=1}^g (\chi, t_0-\alpha_i)_{w_j}=- \sum_{v\in S} Cores(\chi, t_v-\alpha_i)_v =-\sum_{v\in S} f^*(Cores_{k_i/k} (\chi, t-\alpha_i))(x_v) =0$$ by (\ref{hara}) and the functoriality of Brauer-Manin pairing. 
By (\ref{spr}), one obtains that $$(\chi, t_0-\alpha_i)_{w_{j_0}}=0 \ \ \text{ for all $\chi\in H^1(L_i/k_i, \Bbb Q/\Bbb Z)$. } $$  
Since $L_i/k_i$ is abelian, one concludes that $w_{j_0}$ splits completely in $L_i/k_i$. 
There is $$y_{v_i} \in \mathcal X_{t_0}'(\frak o_{k_{v_i}})\subset \mathcal X(\frak o_{k_{v_i}}) \ \ \text{ such that } \ \  
 y_{v_i} \mod v_i \in \mathcal D^{sm}_{i,v_i}(k(v_i)) \subset  \mathcal X_{t_0}'(k(v_i)) $$
by Hensel's lemma and (a) and (c) for $1\leq i\leq n$. 
One can extend these $y_{v_i}$ to $$(y_v)_{v\in \Omega_k} \in X_{t_0}(\mathbf A_k)\cap W $$ by taking any element in $W_v\cap X_{t_0}(k_v)$ for $v\in S$ and any element in $\mathcal X_{t_0}(\frak o_{k_v})$ for the rest of primes $v$ of $k$.
Then
$$\sum_{v\in \Omega_k} b(y_v) = \sum_{v\not\in S} b(y_v) = \sum_{i=1}^n  b(y_{v_i}) $$ by (\ref{hara}) and (d) for all $b\in (B+B_1)$. 
By \cite[Corollary 2.4.3]{Ha94} and the choice of $y_{v_i}$, there is $\tau_i\in \Gal(M_i/k_i)$ such that $$ b(y_{v_i}) =\partial_{D_{i, 1}}(b) (\tau_i) $$
for all $b\in B+B_1$. Therefore  
\begin{equation} \label{res-bm}\sum_{v\in \Omega_k} b(y_v) =  \sum_{i=1}^n \partial_{D_{i,1}} (b) (\tau_i) \in \BQ/\BZ \end{equation}  
with $(\tau_i)_{i=1}^n \in \bigoplus_{i=1}^n \Gal(M_i/k_i)$ for all $b\in (B+B_1)$. 
 Since 
$$ \sum_{v\in \Omega_k} f^*(\xi) (y_v) = \sum_{v\in \Omega_k} \xi (t_0) =0  $$  for all $\xi \in \partial_U^{-1} (\bigoplus_{i=1}^n H^1(M_i/k_i, \BQ/\BZ))$ by the functoriality of Brauer-Manin pairing and the reciprocity law, one obtains $$\chi ((\tau_i)_{i=1}^n)=0 \ \ \ \text{for all} \ \chi \in \bigoplus_{i=1}^n H^1(M_i/k_i, \BQ/\BZ)$$ by (\ref{residue}) and (\ref{res-bm}). 
This implies that $(\tau_i)_{i=1}^n \in \bigoplus_{i=1}^n [\Gal(M_i/k_i), \Gal(M_i/k_i)]$. Write  
$$ [\Gal(L_i/k_i), \Gal(M_i/L_i)]= \{ \prod_{\gamma \in \Gal(L_i/k_i)} (\gamma \circ \sigma_{\gamma} ) \sigma_{\gamma}^{-1} : \sigma_\gamma \in \Gal(M_i/L_i) \}  $$ for $1\leq i\leq n$.
Then $$[\Gal(L_i/k_i), \Gal(M_i/L_i)] \subseteq  [\Gal(M_i/k_i), \Gal(M_i/k_i)]\subseteq \Gal(M_i/L_i) .$$ 
Since $\Gal(L_i/k_i)$ is cyclic, one has $$H^2(L_i/k_i, \Bbb Q/\Bbb Z)=H^3(L_i/k_i, \Bbb Z)=H^1(L_i/k_i, \Bbb Z)=0 .$$
This implies the restriction map 
$$ H^1(M_i/k_i, \Bbb Q/\Bbb Z) \rightarrow H^1(M_i/L_i, \Bbb Q/\Bbb Z)^{\Gal(L_i/k_i)}  $$ is surjective for $1\leq i\leq n$. Therefore the natural map
$$ \Gal(M_i/L_i)/ [\Gal(L_i/k_i), \Gal(M_i/L_i)]  \rightarrow \Gal(M_i/k_i) /[\Gal(M_i/k_i), \Gal(M_i/k_i)] $$ is injective by the duality. One concludes that $$[\Gal(L_i/k_i), \Gal(M_i/L_i)] = [\Gal(M_i/k_i), \Gal(M_i/k_i)]$$ for $1\leq i\leq n$. Then 
\begin{equation} \label{kill}  \tau_i= \prod_{\gamma_i \in \Gal(L_i/k_i)} (\gamma_i \circ \sigma_{\gamma_i}) \sigma_{\gamma_i}^{-1}  \ \ \text{with} \ \ \sigma_{\gamma_i}\in \Gal(M_i/L_i)  \end{equation} 
for $1\leq i\leq n$.

For each pair $(\gamma_i, \sigma_{\gamma_i})\in \Gal(L_i/k_i) \times \Gal(M_i/L_i)$ in (\ref{kill}), there is $\frak q\in T_i$ such that $$\sigma_{\gamma_i}^{-1} = Frob(\frak q) \ \ \ \text{ and } \ \ \ (L_i)_\frak q = k_\frak p  \ \text{ with } \ \frak p=\frak q\cap k$$ for $1\leq i\leq n$. 
Since $\ord_\frak p(p_i(t_0))=1$,
there is $$y_{\frak p}' \in \mathcal X_{t_0}'(\frak o_{k_{\frak p}})\subset \mathcal X_{t_0}(\frak o_{k_{\frak p}}) \ \ \text{ such that } \ \  
 y_{\frak p}' \mod \frak p \in \gamma_i \mathcal D^{sm}_{i,\frak p}(k(\frak p)) \subset  \mathcal X_{t_0}'(k(\frak p)) $$
by Hensel's lemma and (a) and (c) for $1\leq i\leq n$. Define 
$$ z_v = \begin{cases} y_{\frak p}' \ \ \ & \text{ if $v=\frak p$ from a pair $(\gamma_i, \sigma_{\gamma_i})$ in (\ref{kill}) with $1\leq i\leq n$} \\
y_v \ \ \ & \text{otherwise.} \end{cases} $$ Then $(z_v)_{v\in \Omega_k} \in X_{t_0} (\mathbf A_k) \cap W$ by (d). For $\frak p$ from a pair $(\gamma_i, \sigma_{\gamma_i})$ from (\ref{kill}), one has  
\begin{equation} \label{twist-value} b(y_{\frak p}')=e(\frak p) \partial_{D_i} (b) (Frob(y_{\frak p}')) = \ord_{\frak p} (p_i(t_0)) \partial_{D_i}(b) (\gamma_i \circ \sigma_{\gamma_i}^{-1})= -\partial_{D_i}(b) (\gamma_i \circ \sigma_{\gamma_i} ) \end{equation}
and 
\begin{equation} \label{br-value}  b(y_{\frak p}) = e(\frak p) \partial_{D_i} (b) (Frob(y_{\frak p})) = \ord_{\frak p} (p_i(t_0)) \partial_{D_i}(b) (\sigma_{\gamma_i}^{-1} )=- \partial_{D_i}(b) (\sigma_{\gamma_i} ) \end{equation}
for all $b\in (B+B_1)$ by \cite[Corollary 2.4.3 and p.244-245]{Ha94} with $1\leq i\leq n$. Therefore 
 $$ \sum_{v\in \Omega_k} b(z_v) = \sum_{v\in \Omega_k} b(y_v) -\sum_{i=1}^n \sum_{\gamma_i\in \Gal(L_i/k_i)} b(y_{\frak p}) +\sum_{i=1}^n \sum_{\gamma_i\in \Gal(L_i/k_i)} b(y_{\frak p}') $$
$$ = \sum_{i=1}^n \partial_{D_i}(b) (\tau_i) +\sum_{i=1}^n \sum_{\gamma_i \in \Gamma_i} \partial_{D_i}(b) (\sigma_{\gamma_i} ) - \sum_{i=1}^n \sum_{\gamma_i\in \Gal(L_i/k_i)} \partial_{D_i}(b) (\gamma_i\circ \sigma_{\gamma_i} )=0$$ 
for all $b\in (B+B_1)$ by (\ref{kill}), (\ref{twist-value}) and (\ref{br-value}). This implies that 
$$ (z_v)_{v\in \Omega_k} \in (X_{t_0} (\mathbf A_k)_{E_c}^{B+B_1}\cap W ) \neq \emptyset $$ as desired. 
\end{proof}

\begin{rem} The assumption that $L_i/k_i$ is cyclic for $1\leq i\leq n$ is needed only for proving the equality  
$$ [\Gal(M_i/k_i), \Gal(M_i/k_i)]= \{ \prod_{\gamma \in \Gal(L_i/k_i)} (\gamma \circ \sigma_{\gamma} ) \sigma_{\gamma}^{-1} : \sigma_\gamma \in \Gal(M_i/L_i) \}  $$
in  the proof of Proposition \ref{sc-bm}. This equality is equivalent to that the inflation map $$H^2(L_i/k_i, \Bbb Q/\Bbb Z)\rightarrow H^2(M_i/k_i, \Bbb Q/\Bbb Z)$$ is injective. In this case, the assumption that $L_i/k_i$ is  abelian will be enough for Proposition \ref{sc-bm}. 
\end{rem}

The main result of this section is the following theorem.

\begin{thm} \label{main-schinzel} Let $X\xrightarrow{f} \BA_k^1$ be a surjective morphism over a number field $k$. Suppose that $X$ admits an action of a torus $T$ over $k$ such that $f^{-1}(U) \xrightarrow{f} U$ is a torsor under $T$ where $U$ is an open dense subset of $\Bbb A_k^1$ over $k$. 
Write $\Bbb A_k^1\setminus U=\{P_1, \cdots, P_n \}$ where $P_1, \cdots, P_n$ are the closed points over $k$ and $k_i=k(P_i)$ are the residue fields of $P_i$ for $1\leq i\leq n$.
Assume that 

i)  Each $f^{-1}(P_i)$ contains an irreducible component of multiplicity 1 such that the algebraic closure $L_i$ of $k_i$ inside the function field of this component is cyclic over $k_i$ with $1\leq i\leq n$.

ii)  Each equivalent class of admissible connected components of $X(k_\infty)$ contains a connected component $E=\prod_{v\in \infty_k} E_v$ such that all $f(E_v)$'s are unbounded over all real primes $v$.

iii) Hypothesis ($H_s$) holds.

 Then $X$ satisfies strong approximation off $\infty_k$ with respect to $\Br_a(X)$. 
\end{thm}

\begin{proof} It follows from the same proof of Theorem \ref{main-1}  by replacing Proposition \ref{fib-bm} with Proposition \ref{sc-bm}. 
\end{proof}

It should be pointed out that weak approximation for varieties satisfying Theorem \ref{main-schinzel} is not known before (see \cite{Wei1}). 
Applying the above theorem to the norm equation in (\ref{n-eq}), one gets the following result which improves \cite[Theorem 2]{Gun} and \cite[Theorem 1.1]{MIT}. 

\begin{cor}\label{neq-sch} Let $X$ be the smooth part of the following equation
$$ \prod_{i=1}^m N_{L_i/k}(x_i) = c\prod_{j=1}^n p_j(t)^{e_j} \ \ \ \text{with} \ \ \ c\in k^\times $$ 
where $L_i/k$'s are finite extensions of number fields and $p_j(t)$'s are distinct irreducible monic polynomials over $k$ and $e_j$'s are positive integers. Assume that one of extensions $L_i/k$ for $1\leq i\leq m$ is cyclic.
Suppose there is a connected component $E=\prod_{v\in \infty_k} E_v$ in each equivalent class of admissible connected components of $X(k_\infty)$ such that the projection of $t$-coordinate of $E_v$ is unbounded for all real primes $v\in \infty_k$. If Schinzel's hypothesis holds, then $X$ satisfies strong approximation off $\infty_k$ with respect to $\Br_1(X)$.  
\end{cor}

\begin{proof} It follows from the same proof of Corollary \ref{app-neq} by replacing Theorem \ref{main-1} with Theorem \ref{main-schinzel}. \end{proof}

\section{Fibration over $\Bbb P^1$ with an action of torus}

Harpaz and Wittenberg proposed a conjectures in \cite[Conjecture 9.1 and 9.2]{HW} related with 

$\bullet$ a set of distinct irreducible polynomials $\{p_1(t), \cdots, p_n(t)\}$ over $k$ with $k_i=k[t]/(p_i(t))$ and $a_i\in k_i$ representing the class of $t$ for $1\leq i\leq n$;

$\bullet$ a set of finite extensions $\{L_1/k_1, \cdots, L_n/k_n\}$ and $b_i\in k_i^\times$ for $1\leq i\leq n$.  

Consider the morphism
$$  \Phi: \  (\Bbb A_k^2 \setminus \{(0,0)\})\times_k \prod_{i=1}^n(\Res_{L_i/k}(\Bbb A_{L_i}^1)\setminus F_i) \longrightarrow \prod_{i=1}^n \Res_{k_i/k}(\Bbb A_{k_i}^1) $$ given by  $$(\lambda, \mu, x_1, \cdots, x_n)\mapsto (b_i(\lambda-a_i\mu)-N_{L_i/k_i}(x_i)))_{1\leq i\leq n}$$
where $F_i$ is the singular locus of the variety of $\Res_{L_i/k}(\Bbb A_{L_i}^1)\setminus \Res_{L_i/k}(\Bbb G_{m,L_i})$. They define
\begin{equation} \label{w-equ} W=\Phi^{-1}(0,\cdots, 0) \end{equation}
in \cite[\S 9.2.2.]{HW} and show that \cite[Conjecture 9.1]{HW} holds if $W$ satisfies strong approximation off any finite prime (see \cite[Corollary 9.10]{HW}).

In this section, we apply Proposition \ref{assum} to study the fibration over $\Bbb P^1$ with an action of torus and use \cite[Conjecture 9.1 and 9.2]{HW} to establish strong approximation of $W$.

The following result is due to Sansuc \cite[Theorem 4.1]{S1}, which provides the explicit version of Waldschmidt's result in \cite{Wa} and \cite{Wa1}.

\begin{lem}\label{sansuc1}
Let  $\frak m$ be an integral ideal of a number field $k$. Suppose $\{ v_1, \cdots , v_s\}$ is a set of finite primes of $k$ above a set of primes $\{p_1, \cdots, p_s\}$ of $\Bbb Q$ such that  each $p_i$ splits completely in Galois closure of $k/\Bbb Q$ with $v_i\nmid \frak m$ for $1\leq i\leq s$.
If $s>[k:\Bbb Q]^2-[k:\Bbb Q]+1$, then the set
$$ \{ x\in \frak o_{k,S_0}^\times:  \ \  x\equiv 1 \mod \frak m,  \ \  x \ \text{totally positive at real primes} \} $$
 is dense in the connected component of 1 in $(k\otimes_\Bbb Q \Bbb R)^\times$ with $S_0=\{ v_1, \cdots , v_s\}\cup \infty_k$. 
\end{lem}

One can apply Lemma \ref{sansuc1}  to established strong approximation with Brauer-Manin obstruction for tori with the property of approximation at archimedean primes.

\begin{thm}  \label{apa} Let $T$ be a torus over a number field $k$. If $S$ is a finite subset of $\Omega_k$, there is a finite subset $S_0$ of $\Omega_k$ with $S_0\cap S=\emptyset$ and $|S_0|$ independent of $S$ such that $T$ satisfies strong approximation with respect to $\Br_1(T)$ off $S_0$. 
\end{thm}

\begin{proof} We first prove the case that $T$ is a quasi-trivial torus. Without loss of generality, we assume that $T=\Res_{K/k}(\Bbb G_m)$ where $K/k$ is a finite extension. Choose a set of primes $\{p_1, \cdots, p_s\}$ of $\Bbb Q$ with $s>[K:\Bbb Q]^2-[K:\Bbb Q]+1$ such that each $p_i$ splits completely in Galois closure of $K/\Bbb Q$ and any prime in $S$ is not above $p_i$ for $1\leq i\leq s$. Let $S_0$ be a set of all primes of $k$ above $p_i$ for $1\leq i\leq s$. 

For any open subset $\prod_{v\in S_0} T(k_v) \times \prod_{v\not\in S_0} U_v$ in $T( \mathbf A_k)$ with
$$ (\prod_{v\in S_0} T(k_v) \times \prod_{v\not\in S_0} U_v) \cap T( \mathbf A_k)^{Br_1(T)} \neq \emptyset, $$
 there is $x\in T(k)=K^\times$ such that 
$$ (x\cdot T^+(k_\infty)) \cap (\prod_{v\in S_0} T(k_v) \times \prod_{v\not\in S_0} U_v)\neq \emptyset $$ by \cite[Theorem 3.19]{D11}, where $$T^+(k_\infty)=\prod_{v\in \infty_k}T^+(k_v)$$ is the connected component of 1 in $T(k_{\infty_k})=(K\otimes_k k_\infty)^\times $.

Let $S_1$ be a finite subset of $\Omega_k$ with $\infty_k\subseteq S_1$ and $S_1\cap S_0=\emptyset$ such that $$U_v=(\frak o_{K} \times_{\frak o_k} \frak o_{k_v})^\times=\prod_{w\mid v} \frak o_{K_w}^\times $$ for any $v \not\in S_1\cup S_0$.  
Let $\frak m$ be an integral ideal of $\frak o_{K}$ consisting of the primes above $S_1\setminus \infty_k$ such that $$x(1+\frak m)\subseteq U_v$$ for all $v\in S_1\setminus \infty_k$. Let $\tilde{S_0}$ be the set of primes of $K$ above $S_0$ and $\infty_k$. There is $y\in \frak o_{K,\tilde{S_0}}^\times$ with $y\equiv 1 \mod \frak m$ such that $$y\in ((x^{-1}\cdot \prod_{v\in \infty_k} U_v) \cap T^+(k_\infty)) \subseteq T^+(k_\infty)$$ by applying Lemma \ref{sansuc1}. This implies that 
$$xy \in T(k) \cap  (\prod_{v\in S_0} k_v^\times \times \prod_{v\not\in S_0} U_v)$$ as desired. 

For a general torus $T$, one consider a flasque resolution of $T$ in sense of \cite[Proposition-Definition 3.1]{CT08}
$$ 1\rightarrow T_1\rightarrow T_0 \rightarrow T\rightarrow 1 $$ 
where $T_0$ is a quasi-trivial torus over $k$ and $T_1$ is a flasque torus over $k$.  Since $T_0$ satisfies strong approximation with respect to $\Br_1(T_0)$ off $S_0$ as above, one concludes that $T$ satisfies strong approximation with respect to $\Br_1(T)$ off $S_0$ by the decent relation in \cite[Theorem 5.1]{CLX}. 
\end{proof}

\begin{cor} Let $G$ be a connected linear algebraic group over a number field $k$. If $S$ is a finite subset of $\Omega_k$, there is a finite subset $S_0$ of $\Omega_k$ with $S_0\cap S=\emptyset$ and $|S_0|$ independent of $S$ such that $G$ satisfies strong approximation with respect to $\Br_1(G)$ off $S_0$.
\end{cor}

\begin{proof} Since $$1\rightarrow  G^u \rightarrow G\rightarrow  G^{red} \rightarrow 1 $$ where $G^u$ is the unipotent radical of $G$, one can assume $G=G^{red}$ by \cite[Lemma 2.1]{CDX} and \cite[Proposition 4.1]{CLX}. By \cite[Proposition-Definition 3.1]{CT08} and \cite[Theorem 5.1]{CLX}, one can further assume that $G$ is reductive and quasi-trivial. Then one has the following short exact sequence of algebraic groups
$$ 1\to [G, G] \to G \to T \to 1 $$ where $[G, G]$ is semi-simple and simply connected and $T$ is quasi-trivial. Enlarge $S_0$ such that each simple factor of $[G,G]$ is isotropic over $S_0$  with $S_0\cap S=\emptyset$. Since $[G, G]$ satisfies strong approximation off $S_0$ by \cite[Theorem 7.12]{PR}, one obtains the desired result by Theorem \ref{apa}, \cite[Proposition 4.1]{CLX}, \cite[Lemma 2.1]{CDX} and \cite[Theorem 5.1.1(e)]{sko}.
\end{proof}

The main result of this section is the following result.

\begin{thm} \label{t} 
Let $X$ be a smooth variety with an action of  a torus $T$ over a number field $k$. Suppose that $X\xrightarrow{f} \Bbb P^1$ is a morphism over $k$ such that all fibres of $f$ contain an irreducible component of multiplicity 1. Assume that there is an open dense subset $U\subset \Bbb P^1$ over $k$ such that $f^{-1}(U)\xrightarrow{f} U$ is a torsor under $T$ and $\Bbb P^1\setminus U=\{P_1, \cdots, P_n\}$ contains a rational point. 

(1) If Conjecture 9.1 in \cite{HW} is true for $\{P_1, \cdots, P_n\}$, then for any finite subset $S$ of $\Omega_k$ there is a finite subset $S_0$ of $\Omega_k$ with $S\cap S_0=\emptyset$ and $|S_0|$ independent of $S$ such that $X$ satisfies strong approximation with respect to $\Br_1(X)$ off $S_0$.

(2) If Conjecture 9.1 in \cite{HW} is true for $\{P_1, \cdots, P_n\}$, then $T(k_\infty)^+ \cdot f^{-1}(U)(k)$ is dense in $X(\RA_k)^{\Br_1(X)}$ where $T(k_\infty)^+$ is the connected component of identity of Lie group $T(k_\infty)$.

(3) If Conjecture 9.2 in \cite{HW} is true for $\{P_1, \cdots, P_n\}$, then $f^{-1}(U)(k)$ is dense in $X(\RA_k)_\bullet^{\Br_1(X)}$. 
\end{thm}

\begin{proof}  By Theorem \ref{apa}, there exists a finite subset $S_0$ of $\Omega_k$ with $S\cap S_0=\emptyset$ such that $T(k_\infty)^+ $ can be approximated by the elements in $T(k)\cdot \prod_{v\in S_0} T(k_v)$. 
The statement (2) gives 
  $$X(\RA_k)^{\Br_1(X)}=\overline{T(k_\infty)^+ \cdot f^{-1}(U)(k)} \subset \overline{(\prod_{v\in S_0}T(k_v))\cdot X(k)}\subset X(\RA_k) .$$
This implies that the statement (1) holds. We only need to show (2) and (3).

Let $E$ be a connected component of $X(k_\infty)$ and $W$ be an open compact subset of $X(\mathbf A_k^f)$ such that 
$$ (E\times W) \cap X(\RA_k)^{\Br_1(X)}\neq \emptyset . $$
Since $\Bbb P^1\setminus U=\{P_1, \cdots, P_n\}$ contains a rational point, one concludes that $U$ is quasi-trivial. By Proposition \ref{assum}, there is a finite subgroup $B\subset \Br_1(f^{-1}(U))$ such that 
$$ (E\times W)\cap X_c(\mathbf A_k)^B \neq \emptyset  \ \ \ \Longleftrightarrow \ \ \ (E\times W) \cap X_c(\mathbf A_k)^{\Br_a(X_c)} \neq \emptyset $$ for all $c\in U(k)$.

For (2),  there is $u\in U(k)$ such that 
$$X_u(\RA_k)^{B}\cap (E\times W) \neq \emptyset $$  by \cite[Theorem 9.17]{HW}. Therefore $$(E\times W) \cap X_u(\mathbf A_k)^{\Br_a(X_u)} \neq \emptyset .$$
This implies that $X_u$ is a trivial torsor under $T$ over $k$ by \cite[Theorem 5.2.1]{sko}.  Moreover, there are $g_\infty\in T(k_\infty)^+ $ and $x\in X_u(k)\subset f^{-1}(U)(k)$ such that $g_\infty \cdot x\in W$ by \cite[Theorem]{Ha08}.

For (3), one also can apply \cite[Theorem 9.17]{HW}.  The properness of $f$ is only used to check that a connected component of $f^{-1}(U)(\Bbb R)$ maps onto a connected component of $U(\Bbb R)$ in \cite[pp. 281, line 12]{HW}. This is true in our situation by Hilbert 90. Then there is $u\in U(k)$ such that  
$$X_u(\RA_k)_\bullet^B\cap (E\times W) \neq \emptyset .$$
By the same argument as above (2), one concludes 
$$[(f^{-1}(U)(k)\cap (E\times W)]\supset [X_u(k) \cap (E\times W) ]\neq \emptyset$$ as desired. 
\end{proof}

Consider the surjective homomorphism of tori
$$\psi: \BG_m\times \prod_{i=1}^n \Res_{L_i/k}\BG_m\rightarrow \prod_{i=1}^n \Res_{k_i/k}\BG_m; \ \ \ (\alpha, \alpha_1,\cdots, \alpha_n )\mapsto ( \alpha^{-1} \cdot N_{L_i/k_i}\alpha_i)_{1\leq i\leq n} $$
where $N_{L_i/k_i}$ is the norm map for $1\leq i\leq n$. Then $\ker(\psi)$ acts on $W$ by 
$$ (\alpha, \alpha_1, \cdots, \alpha_n) \circ (\lambda, \mu, x_1, \cdots, x_n) = (\alpha\cdot \lambda, \alpha\cdot \mu, \alpha_1 \cdot x_1, \cdots, \alpha_n \cdot x_n) . $$

Let $f: W\rightarrow \Bbb P^1$ be the fibration given by
$ (\lambda, \mu, x_1, \cdots , x_n ) \mapsto  [\lambda, \mu] $ and $$U=\Bbb P^1\setminus \{(p_1(t)), \cdots, (p_n(t)) \}  $$ where $p_i(t)$ with $1\leq i\leq n$ are the irreducible polynomials with $\lambda = t\mu$. Then 
$f^{-1}(U) \rightarrow U$ is a torsor under $\ker(\psi)$.

\begin{lem} \label{w} Let $W$ be defined as (\ref{w-equ}). Then 
$$ \bar k[W]^\times = \bar k^\times, \ \ \ \Pic(W_{\bar k})=0 \ \ \ \text{and}\ \ \ \Br_1(W)=\Br(k) $$
where  $\bar k$ is an algebraic closure of $k$.
\end{lem}

\begin{proof} Consider the variety $W_1$ over $\bar k$ defined by the following equations 
\begin{equation} \label{nequ}  \prod_{j=1}^{d_i} z_{i,\sigma_i}^{(j)}= \sigma_i(b_i)(\lambda-\sigma_i(a_i)\mu)  \end{equation}
where $\sigma_i$ runs over all embeddings from $k_i/k$ to $\bar k/k$ and $d_i=[L_i:k_i]$ for $1\leq i\leq n$.  For any fixed $\sigma_i$ with $1\leq i\leq n$, the variety defined by a single equation of (\ref{nequ}) is isomorphic to $\Bbb A_{\bar k}^{d_i+1}$. This implies  
 \begin{equation}  \label{invertible} (\bar k[z_{i, \sigma_i}^{(1)}, \cdots, z_{i, \sigma_i}^{(d_i)}, \lambda, \mu]/(\prod_{j=1}^{d_i} z_{i, \sigma_i}^{(j)}-\sigma_i(b_i)(\lambda-\sigma_i(a_i)\mu)))^\times =\bar k^\times . \end{equation}
 We claim that $\bar k[W_1]^\times =\bar k^\times$.  Indeed, suppose $u\in (\bar k[W]^\times \setminus \bar k^\times)$, there is $\sigma_{i_0}$ with $1\leq i_0\leq n$ such that the specialization of $u$ with $z_{i, \tau}^{(j)}\in \bar k$ for all  $\tau\neq \sigma_{i_0}$ and $1\leq j\leq d_i$ is not a constant. This contradicts to (\ref{invertible}) and the claim follows. Since $W_{\bar k}$ is obtained by removing a closed subset of $W_1$ of codimension bigger than 1, one concludes that $\bar k[W]^\times =\bar k [W_1]^\times = \bar k^\times$. 

Since $\Pic(U_{\bar k})=\Pic(\ker(\psi)_{\bar k})=0$, one obtains that $\Pic (f^{-1}(U)_{\bar k})=0$ and 
$$ 1\rightarrow \bar k[U]^\times \rightarrow \bar k[f^{-1}(U)]^\times \rightarrow \bar k[\ker(\psi)]^\times \rightarrow 1 $$
by \cite[Prop.6.10]{S}. Therefore 
$$ \Pic(W_{\bar k}) =\coker(\bar k[f^{-1}(U)]^\times \xrightarrow{div} \Div_{W_{\bar k}\setminus f^{-1}(U)_{\bar k}} W_{\bar k} ) $$ 
and $\{ z_{i,\sigma_i}^{(j)}\}$ is a basis of the free abelian group $\bar k[f^{-1}(U)]^\times/\bar k^\times$ where $\sigma_i$ runs over all embeddings from $k_i/k$ to $\bar k/k$ and $1\leq j\leq d_i$ for $1\leq i\leq n$.  Since $W\xrightarrow{f} \Bbb P^1$ is faithful flat, one gets that $\Div_{W_{\bar k}\setminus f^{-1}(U)_{\bar k}} W_{\bar k} $ is a free abelian group with a basis $\{f^{-1}(\lambda-\sigma_i(a_i)\mu)\}$ where  $\sigma_i$ runs over all embeddings from $k_i/k$ to $\bar k/k$ with $1\leq i\leq n$.  By the equation (\ref{nequ}), one concludes that $\Pic(W_{\bar k})=0$. By the Hochschild-Serre spectral sequence (see \cite[Lemma 6.3 (i)]{S}), one has $\Br_1(W)=\Br(k)$. 
\end{proof}

\begin{cor}  Let $W$ be defined by (\ref{w-equ}) and one of polynomials $\{ p_1(t),\cdots, p_n(t) \}$ be of degree one. 

(1)  If \cite[Conjecture 9.1]{HW} holds for $\{ p_1(t),\cdots, p_n(t) \}$, then for any finite subset $S$ of $\Omega_k$, there is a finite subset $S_0$ of $\Omega_k$ with $S\cap S_0=\emptyset$ and $|S_0|$ independent of $S$ such that $W$ satisfies strong approximation off $S_0$.

(2) If \cite[Conjecture 9.2]{HW} holds for $\{ p_1(t),\cdots, p_n(t) \}$, then  $W(k)$ is dense in $W({\bf A}_k)_{\bullet}$.
\end{cor}

\begin{proof}  The results follow from Theorem \ref{t} and Lemma \ref{w}.
\end{proof}

\noindent\textbf{Acknowledgements.} This work has benefited a lot from the program "Reinventing rational points" in IHP. 
We would like to thank J.-L. Colliot-Th\'el\`ene for his useful comments and Y. Harpaz and O. Wittenberg for pointing out an error in Theorem \ref{t} on the early version of this paper.  The first named author is supported by Alexander von Humboldt foundation, the second and the third named authors are supported by NSFC grant no.11631009.

\bibliographystyle{alpha}
\end{document}